\documentclass[times,sort&compress,3p]{elsarticle}
\journal{Journal of Multivariate Analysis}
\usepackage[labelfont=bf]{caption}

\usepackage{amsmath,amsfonts,amssymb,amsthm,booktabs,color,epsfig,graphicx,hyperref,url}

\theoremstyle{plain}

\newtheorem{proposition}{Proposition}
\newtheorem{lemma}{Lemma}

\theoremstyle{definition}
\newtheorem{definition}{Definition}
\newtheorem{remark}{Remark}
\newtheorem{example}{Example}

\def \d{\frac{1}{2}\, }\def\<{\langle}\def\>{\rangle}\def\tr{\mathrm{trace}\, }
\def\S{\sum_{n=0}^{\infty}}
\def \E{\mathbb{E}}\def \N{\mathbb{N}}\def \R{\mathbb{R}}\def \C{\mathbb{C}}\def \Z{\mathbb{Z}}
\def\1{\vec{\mathbf{1}}}
\begin{document}
\begin{frontmatter}

\title{Duality for real and multivariate exponential families}

\author[1]{G\'erard Letac
\corref{mycorrespondingauthor}}

\address[1]{Institut de Math\'ematiques de Toulouse, 118 route de Narbonne 31062 Toulouse, France.}

\cortext[mycorrespondingauthor]{Corresponding author. Email address: gerard.letac@math.univ-toulouse.fr \url{}}

\begin{abstract}
Consider a measure $\mu$ on $\R^n$ generating a natural  exponential family $F(\mu)$  with variance function $V_{F(\mu)}(m)$
and  Laplace transform $$ \exp(\ell_{\mu}(s))=\int_{\R^n} \exp(-\<s,x\>)\mu(dx).$$ A dual measure $\mu^*$   satisfies $-\ell'_{\mu^*}(-\ell'_{\mu}(s))=s.$ Such a dual measure does not always exist.
One important property is $\ell''_{\mu^*}(m)=(V_{F(\mu)}(m))^{-1},$ leading to the notion of duality  among exponential families (or rather among the extended notion of T exponential families $T\hskip-2pt F$ obtained by considering all translations of a given exponential family $F$).
\end{abstract}


\end{frontmatter}
\section{Introduction\label{sec:1}}
One can be surprized by the explicit  formulas that one gets from large deviations in one dimension. Consider iid random variables $X_1,\ldots, X_n, \ldots$  and  the empirical mean $\overline {X}_n=(X_1+\cdots+X_n)/n.$ Then the Cram\'er theorem \cite{CRAMER} says that the limit $$\Pr(\overline {X}_n>m)^{1/n}\to_{n\to \infty}\alpha(m)$$ does exist for $m>\E(X_1).$  For instance,  the symmetric Bernoulli case $\Pr(X_n=\pm1)=\d$ leads, for $0<m<1$, to  \begin{equation}\label{SURPRISE}\alpha(m)= (1+m)^{-1-m}(1-m)^{-1+m}.\end{equation}
For  the  bilateral exponential  distribution $X_n\sim \d e^{-|x|}dx$ we get, for $m>0$,
$$\alpha(m)=\d e^{1-\sqrt{1+m^2}}(1+\sqrt{1+m^2}).$$ What strange formulas! Other ones can be found in \cite{LETAC1995}, Problem 410.  The present paper is going to interpret in certain cases the function $m\mapsto 1/\alpha(m)$  as the Laplace transform of a certain dual measure $\mu^*$ which can be deduced explicitly from the  distribution $\mu$ of $X_n.$ Since the Cram\'er theorem  can be seen as  a result about one dimensional exponential families,  we develop the idea in this framework, using the large box of examples obtained   from the theory of variance functions of exponential families initiated by the article of Carl Morris \cite{MORRIS} in 1982. An even simpler example of duality is provided by the Tweedie families (Barlev and Enis \cite{BARLEV}, J\o rgensen \cite{JORGENSEN1987} and  Tweedie \cite{TWEEDIE}): the variance function $Am^p$ with $p>1$ has dual $Bm^q$ with $\frac{1}{p}+\frac{1}{q}=1$  for suitable pairs $(A,B)$ (see \eqref{ABPAIR}). For instance, the Inverse Gaussian family $Am^3$ has dual $Bm^{3/2}$ .

The  Poisson distribution, with the one of the simplest  variance functions $V_F(m)=m$, leads to the study of exponential family with variance function $e^m$ generated up to a translation by the unsymmetric stable law with Laplace transform $e^{-s}s^s$, also called Landau distribution. This  gives tools for describing duals of other familiar exponential families. The  cases of the normal and gamma families are very simple, being self dual, but  other familiar cases like the negative binomial, the Bernoulli distribution and  the cubic families are tougher. Finally, we consider another  family, which  is the dilogarithm family with variance function $e^m-1$ as well as the one with variance function $\sinh m.$ Like the normal and gamma families, they have  the remarkable property to be self dual.

The definition of dual measures makes sense also in $\R^n$ while the probabilistic interpretation in terms of large deviations is lost. However we consider several cases in $\R^n$: the multinomial distribution, the Wishart ones and other quadratic families as classified by Casalis (\cite{CASALIS}).

We proceed as follows: the notion of duality leads us unfortunately to change a bit the tradition about the exponential families.  Indeed we will   use $e^{-sx}$ instead of $e^{\theta x}$  in order to obtain later more readable formulas. This is explained in Section 2, together with the description of the classical objects attached to an exponential family.

 In the preceding lines we have been  vague about duality. Section 3 gives proper definitions, explaining what we call \textit{a} dual measure $\mu^*$ of $\mu$ and  showing that some measures have no dual. We explain also what  a T exponential family $T\hskip-2pt F$ is.   It  is nothing but an exponential family $F$ plus all its  translations. Indeed, talking about the dual $F^*$ of an exponential family $F$ does not exactly make sense, while the dual $T\hskip-2pt F^*$ of $T\hskip-2pt F$ does.  Section 3 gives also the link with large deviations.

Section 4 concentrates on the $T\hskip-2pt F$ when the variance function $V_F(m)$ is $e^m$  and some parent  distributions. It also give details on what we call
L\'evy measures of types 0, 1 and 2. Of course, large parts of this material are well known from probabilists and statisticians: exponential families, variance functions, L\'evy measures, Landau distribution. It was necessary to expose them again for commodity of reading. This section contains crucial calculations for the sequel in Proposition \ref{prop:4.1}.

Section 5 applies the results of Section 4 to the description of the duals of the Morris and the cubic families, with the surprizing fact that they exist all with the only exception of the hyperbolic family with variance function $m^2+1$.

Section  6 describes the self dual dilogarithm distribution $\mu$ on the set $\N=\{0,1,\ldots\}$ of integers defined by
$$\sum_{n=0}^{\infty} \mu(n)z^n=\exp\left(\sum_{k=1}^{\infty}\frac{1}{k^2}(z^k-1)\right)$$ which, for $m>0$, generates the exponential family with variance function $e^m-1$. Since the consideration of this exponential family and of a set of parent distributions is not done in the literature, we develop some of their properties, somewhat deviating from the study of duality. For instance, if  $N$ is the standard Gaussian distribution, then the variance of the exponential family generated by the convolution $N*\mu$  is $e^m+1.$

Section 7 considers the $\R^n$ case:  The multinomial distribution has a very explicit dual expressed in terms of the Landau distribution. The Wishart distribution is self dual as the one dimensional gamma distribution. We prove some negative results, like the fact that the multivariate negative binomial law has no dual.

Section 8 discusses open problems.

\section{Laplace and bilateral Laplace transforms}\label{sec2}

At first the exponential families and Laplace transforms are considered.
A certain tradition among statisticians (see Morris \cite{MORRIS}) as opposed to physicists, and may be to probabilists,  defines  the Laplace transform of  a positive measure $\mu$ on $\R$  and $\R^n$ as follows:
\begin{equation}\label{LAPLACE}L_{\mu}(\theta)=\int_{\R^n}e^{\<\theta ,x\>}\mu(dx).\end{equation} From the H\"older inequality the set $D(\mu)=\{\theta\, ;\, L_{\mu}(\theta)<\infty\}$ is a convex set, and  the function $k_{\mu}=\log L_{\mu}$ is convex on $D(\mu)$.  Actually $k_{\mu}$ is strictly convex outside of the particular case where $\mu$ is concentrated on one point in the case of $\R$, or on an affine hyperplane in the case of $\R^n$.  To avoid trivialities one introduces the interior $\Theta(\mu)$ of $D(\mu).$ One calls
$ \mathcal{M}(\R^n)$ the set of $\mu$ which are not concentrated on an affine hyperplane and such that  $\Theta(\mu)$ is not empty. Such a $\mu$  generates a set of probabilities $$P(\theta,\mu)(dx)=e^{\<\theta ,x\>-k_{\mu}(\theta)}\mu(dx)$$ and $F=F(\mu)=\{P(\theta,\mu)\, ; \, \theta\in \Theta(\mu)\}$ is called the natural exponential family generated by $\mu.$ Note that $\mu$ is not necessarily bounded: simple examples on $\R$  like $\mu(dx)=1_{(0,\infty)}(x )dx$ or $\sum_{n=0}^{\infty}\delta_n$ generate the important families of  exponential distributions  or   geometric discrete laws. Omiting 'natural', we will say always exponential family for short. Objects linked to $F$ are the mean $m$ of $P(\theta,\mu)$, the domain of the means $M_F$ and the inverse function $\psi_{\mu}.$ They are defined by
$$m=k'_{\mu}(\theta)=\int_{\R^n}xP(\theta,\mu)(dx),\qquad  M_F=k'_{\mu}(\Theta(\mu)),\qquad \theta=\psi_{\mu}(m).$$
Note that since $k_{\mu}$ is strictly convex, then $k'_{\mu}$ is injective  on $\Theta(\mu)$, the map $\psi_{\mu}$ from $M_F$ onto
 $\Theta(\mu)$ is well defined and $M_F$ is a connected  open set.

If $C_F$ is  the closed convex set  generated by the support of $\mu$ clearly $M_F$
is contained in $C_F.$  We say that $F$ or $\mu$ are steep if $M_F$ is equal to the interior of $C_F.$ Most of the classical exponential families are steep, but not always (see for instance the Tweedie scale below for $p<0$). In $\R$ the set $M_F$ is an interval, but in $\R^n$  there are non steep examples such that $M_F$ is non-convex (\cite{RIO}, p. 35).

Finally, the last important object about the exponential family $F$ is its variance function $V_F $ defined on $M_F$ by
$$V_F(m)=\frac{1}{\psi'_{\mu}(m)}=k_{\mu}''(\psi_{\mu}(m))=\int_{\R^n}(x-m)\otimes (x-m)P(\psi_{\mu}(m),\mu)(dx),$$
which characterizes $F$.

Now bilateral Laplace transforms are considered. In the particular case of dimension one, an older tradition associates the name of Laplace to integrals $\int_0^{\infty}e^{-sx}\mu(dx)$ which are conveniently defined for $s>0$ in many circumstances. In the present paper we need to consider what the physicists call the bilateral Laplace transform
\begin{equation}\label{BILATERAL}B_{\mu}(s)=\int_{\R^n}e^{-\<s, x\>}\mu(dx).\end{equation}
Dealing with this slight change of notation $B_{\mu}(s)=L_{\mu}(-s)$ will much simplify the description of duality between two natural exponential families. In the sequel,  we say 'Laplace transform' for short and  by \textit{abus de langage} instead of the longer term 'bilateral Laplace transform'.

Because we will deal with these bilateral Laplace transforms,  we have to  modify the description of the classical objects associated to $F=F(\mu)$ with
\begin{eqnarray}\label {MODIF1}S(\mu)&\hskip-8pt=&\hskip-8pt-\Theta(\mu), \qquad \ell_{\mu}(s)= k_{\mu}(-s),\qquad  m= -\ell'_{\mu}(s),\\ \label {MODIF2}s=\varphi_{\mu}(m) &\hskip-8pt=&\hskip-8pt- \psi_{\mu} (m),\qquad \ell''_{\mu}(s)= k''_{\mu}(-s),\qquad V_F(m)=-(\varphi_{\mu}'(m))^{-1}.\end{eqnarray}
Let us insist on the fact that $\ell_{\mu}$ is convex and thus $s\mapsto m=- \ell'_{\mu}(s)$ is decreasing when $n=1$.  For any $n$ the inverse function $m\mapsto s=
\varphi_{\mu}(m)$ from $M_F$ onto $S(\mu)$ is well defined. Thus we have several ways of coding a member of an exponential family:
$$P(\theta,\mu),\qquad  P(-s,\mu), \qquad P(\psi_{\mu}(m),\mu),\qquad P(-\varphi_{\mu}(m),\mu).$$
Finally, in one dimension, to find a generating  measure $\mu$ of the exponential family $F$ from the knowledge of $V_F,$ we proceed as follows:
$$ds= -\varphi'_{\mu}(m)dm=\frac{dm}{V_F(m)}\Rightarrow -s=\int \frac{dm}{V_F(m)}.$$ The choice of the integration constants will change $\mu$ and is arbitrary. In practical cases, we choose these constants in order to get the simplest form of $\mu$. Therefore this choice may depend on aesthetic considerations. Here is an important example in dimension one:

\begin{example}[the Tweedie scale]
This term describes a set of exponential families with variance functions of the form $V_F(m)=\frac{m^p}{\lambda^{p-1}}$ on $M_F=(0,\infty)$ where $\lambda>0$ and $p\in \R\setminus [0,1]$ (see \cite{{BARLEV}, {JORGENSEN1987},{TWEEDIE}}). We are going to describe their densities or Laplace transforms. The limiting cases $p=0$ and $p=1$ correspond respectively to the Gaussian exponential family  with fixed variance $\lambda$ where $M_F=\R$ and to the Poisson family and they need a special treatment.
\begin{enumerate}
\item The stable subordinator case $p>2.$ For simplification we introduce $\alpha=\frac{p-2}{p-1}\in (0,1).$
We write
$$ds=-\frac{\lambda^{p-1}dm}{m^p}\Rightarrow s=\frac{1}{p-1}\left(\frac{\lambda}{m}\right)^{p-1}\Rightarrow m= \frac{\lambda s^{-1/(p-1)}}{(p-1)^{1/(p-1)}}=-\ell'_{\mu}(s).$$ Here $S(\mu)=(0,\infty).$
We finally obtain $\ell_{\mu}(s)=-\lambda\frac{(p-1)^{\alpha}}{p-2}s^{\alpha}.$ This family is generated by a stable law of parameter $\alpha$  with L\'evy measure concentrated on $(0,\infty)$ of Type 1 (see Section 4.3 for the definition of a L\'evy measure of an infinitely divisible probability and its type).
\item The gamma case $p=2.$ Similarly we obtain   $S(\mu)=(0,\infty),$  $\ell'_{\mu}(s)=-\lambda /s$ and $B_{\mu}(s)=1/s^{\lambda}.$ This family is generated by the measure $\mu(dx)$ with density $x^{\lambda-1}/\Gamma(\lambda)$ and is the family of gamma distributions with shape parameter $\lambda$ .

\item The Poisson -gamma case $1<p<2.$ For simplification denote $\beta=\frac{2-p}{p-1}\in (0,\infty).$ A computation analogous to the case $p>2$
leads to \begin{equation}\label{POISSONGAMMA}\ell_{\mu}(s)=\lambda\frac{(p-1)^{-\beta}}{2-p}\frac{1}{s^{\beta}}, \qquad \ell'_{\mu}(s)=-\lambda\frac{(p-1)^{-\frac{1}{p-1}}}{2-p}\frac{1}{s^{\frac{1}{p-1}}}\end{equation} and to  $S(\mu)=(0,\infty).$ This implies that  the corresponding exponential family is the set of laws of $X_1+\cdots +X_{N(t)}$ where $X_1,\ldots,X_n,\ldots$ are iid  gamma distributed with shape parameter $\beta$ and where $N(t)$ is an independent Poisson distribution with mean $t=\lambda\frac{(p-1)^{-\beta}}{2-p}.$

\item The non steep stable case $p<0.$ For simplicity denote $ q=-p>0$  and $\gamma=\frac{q+2}{q+1}\in (1,2).$ Here $S(\mu)=(-\infty,0)$ and we obtain $$\ell_{\mu}(s)=\lambda\frac{(q+1)^{\gamma}}{q+2}(-s)^{\gamma}.$$ The members of this family are stable laws with parameter $ \gamma$ with L\'evy measure concentrated on $(-\infty,0).$ The support of such a stable law is $\R$ but $M_F=(0,\infty):$ this is an example of a non steep family and of an infinitely divisible distribution of Type 2 (see Section 4 and \eqref{STABLE}).

\end{enumerate}
\end{example}
\section{Duality}\label{sec3}
\subsection{Duality between measures}\label{ssec3}
Given $\mu\in \mathcal{M}(\R^n)$ the function $s\mapsto -\ell'_{\mu}(s)=m$ maps $S(\mu)$ onto the domain of the means $M_F$. Its inverse $m\mapsto s=\varphi_{\mu}(m)$  exists and maps $M_F$ onto  $S(\mu).$ Suppose that there exists $\mu^*\in \mathcal{M}(\R^n)$ such that 
\begin{equation}\label{DUALMEASURE}
-\ell'_{\mu^*}(-\ell_{\mu}(s))=s.
\end{equation}
If \eqref{DUALMEASURE} holds we say that $\mu^*$ is a dual measure of $\mu.$ Note that we say 'a dual measure' since $\mu^*$ is unique only up to a multiplicative constant.  However the associated exponential family $F(\mu^*)$ will not change if $\mu^*$ is replaced by $C\mu^*.$  Observe also that if $\mu^*$ exists and if $\mu$ is steep then $\mu$ is also a dual measure of $\mu^*.$   In general, if $\mu^* $ is bounded it is natural to  choose the multiplicative constant such that $\mu^*$ is a probability.

\begin{proposition}\label{prop:3.1} 
Let $\mu\in \mathcal{M}(\R^n)$ and suppose that there exists a dual measure $\mu^*.$  Then
$\ell''_{\mu^*}(m)=(V_{F(\mu)}(m))^{-1}$. Furthermore, if $\mu$ is steep $\ell''_{\mu}(m)=(V_{F(\mu^*)}(m))^{-1}$.
\end{proposition}
\begin{proof}[\textbf{\upshape Proof:}] 
Since $s=\varphi_{\mu}(-\ell'_{\mu}(s))$ we get by definition that
$$-\varphi'_{\mu}(-\ell'_{\mu}(s))=-(-\ell''_{\mu}(s))^{-1}=(V_{F(\mu)}(-\ell'_{\mu}(s)))^{-1},\ \ \ell''_{\mu^*}(m)=(V_{F(\mu)}(m))^{-1}.$$
The second formula is obtained by symmetry. 
\end{proof}
In the particular case of dimension  one, we will have numerous examples of dual measures given below. Let us give the simplest now, based on the Tweedie scale. Essentially, duality exchanges cases 1 and 3.
\begin{proposition}\label{prop:3.2}
Let $1<p<2$ and  $q$ defined by $\frac{1}{p}+\frac{1}{q}=1.$ Let $\mu$, defined by \eqref{POISSONGAMMA} generating the exponential family with variance function $V_F(m)=\frac{m^p}{\lambda^{p-1}}.$ Then there exists a dual measure $\mu^{*}$ and it generates the exponential family with variance function $V_{F(\mu^*)}(s)=(\frac{p}{q})^q\frac{1}{\lambda}s^q.$
\end{proposition}
\begin{proof}[\textbf{\upshape Proof:}]
Using  formula \eqref{POISSONGAMMA} we get
$\ell_{\mu}''(s)=\lambda (p-1)^{-q}s^{-q}$. Since $p-1=p/q$  Proposition \ref{prop:3.1} gives the result. 
\end{proof}
\begin{remark}
For having a more symmetric form we can choose $\lambda$ such that
\begin{equation}\label{ABPAIR}V_{F(\mu)}(m)=\frac{R^p}{q^p}m^p,\ \ V_{F(\mu^*)}(s)=\frac{R^{-q}}{p^q}s^q.\end{equation}
\end{remark}

Dual measures do not always exist. Consider, for instance, Case 4 of the  Tweedie scale above generated by a stable distribution $\mu$ with parameter $\gamma\in (1,2)$ with L\'evy measure concentrated on  $(-\infty,0)$  and domain of the means $M_{F(\mu)}=(0,\infty)$. Therefore if $\mu^*$ was existing we would have  for positive constants $C,C_1,C_2:$
\begin{eqnarray*}\ell_{\mu}(s)&\hskip-8pt =&\hskip-8pt C(-s)^{\gamma}, \qquad \ell'_{\mu}(s)=-C\gamma(-s)^{\gamma-1}=-m,\\ \ell'_{\mu^*}(m)&\hskip-8pt =&\hskip-8pt C_1m^{\frac{1}{\gamma-1}}, \qquad \ell_{\mu^*}(m)=C_2m^{\frac{\gamma}{\gamma-1}}.\end{eqnarray*} This would imply that $\mu^*$ would be a stable law with parameter $\gamma/(\gamma-1)>2$, which is impossible.

Moreover, it is well known that a probability $P\in \mathcal{M}(\R)$ is infinitely divisible if and only if there exists $\nu \in \mathcal{M}(\R)$ such that for $s\in S(P)$ we have
\begin{equation}\label{INFDIV}\ell''_P(s)=\int_{\R}e^{-sx}\nu(dx).\end{equation}
Extension of the definition of infinite divisibility to measures $ \mu\in \mathcal{M}(\R)$ (even the unbounded ones) is easy and is done in Letac \cite{RIO} and the same remark about  $\ell''_{\mu}(s)$ holds. Then we have the following simple proposition:

\begin{proposition}\label{prop:3.3}
Let $ \mu\in \mathcal{M}(\R)$ and assume that a dual measure $\mu^*$ does exist.
Then $\mu^*$ is infinitely divisible if and only if
$m\mapsto \frac{1}{V_{F(\mu)}(m)}$ is the Laplace transform of some measure $\nu\in \mathcal{M}(\R).$
\end{proposition}
\begin{proof}[\textbf{\upshape Proof:}]
This an immediate consequence of Proposition \ref{prop:3.1} and of \eqref{INFDIV}. 
\end{proof}
\begin{example}[Bernoulli distribution on  $\pm 1$]
Proposition \ref{prop:3.3} gives  a powerful tool to check quickly whether a dual measure exists.
Consider, for instance, the Bernoulli measure on $\pm 1$ with mean $m\in (-1,1)$:
$$\mu=\d((1-m)\delta_{-1}+(1+m)\delta_{1}).$$ Therefore the variance of $X\sim \mu$ is $V_{F(\mu)}=\E(X^2)-\E(X)^2=1-m^2.$ To prove that $\mu^*$ exists we observe that
$$\frac{1}{1-m^2}=\d\left(\frac{1}{1-m}+\frac{1}{1+m}\right)=\d\int_{\R}e^{-my-|y|}dy$$
and we apply Proposition \ref{prop:3.3} to the bilateral exponential law $\nu(dy)=\d e^{-|y|}dy.$ One  can prove that $B_{\mu^*}(m)=(1+m)^{1+m}(1-m)^{1-m}$ on $S(\mu^*)=(-1,1)$. We will describe $\mu^*$ later on in Section 4, 
the second formula in \eqref{LAPLACELANDAU} and Section 5.
\end{example}

\subsection{The role of linear transformations and  the J{\o}rgensen set}

Consider the image $\mu_a$ of $\mu$ by the  isomorphism $x\mapsto a(x)$ of $\R^n$ into itself.  Here $a$ is an invertible matrix of order $n$. If $\mu$ has a dual $\mu^*$ then  $(\mu^*)_{a^{-1}}$ is a dual of $\mu_a.$ Let us give the details of a tedious calculation:
\begin{eqnarray*}
&&B_{\mu_a}(s)=\int_{\R^n}e^{-\<s,a(x)\>}=\int_{\R^n}e^{-\<a^Ts,x)\>}\mu(dx),\qquad
\ell_{\mu_a}(s)=\ell_{\mu}(a^Ts),\qquad m=-\ell'_{\mu_a}(s)=-a^T\ell'_{\mu}(a^Ts),\\ &&(a^{-1})^Tm=-\ell'_{\mu}(a^Ts),\qquad a^Ts=\ell'_{\mu^*}((a^{-1})^Tm),\qquad  s=(a^{-1})^T\ell'_{\mu^*}((a^{-1})^Tm), \ell_{(\mu_a)^*}(m)=\ell_{\mu^*}((a^{-1})^Tm),\\&&  B_{(\mu_a)^*}(m)=B_{\mu^*}((a^{-1})^Tm)=\int_{\R^n}e^{-\<a^{-1})^Tm,y\>}\mu^*(dy)=\int_{\R^n}e^{-\<m,a^{-1}y\>}\mu^*(dy)=B_{\mu_{a^{-1}}}(m).
\end{eqnarray*}

Recall also that  the image of $F=F(\mu)$ by the isomorphism $a$ satisfies  $$V_{F(\mu_a)}(m)=a^TV_F(a^{-1}m)a.$$

If $\mu\in \mathcal{M}(\R^n)$ then the J{\o}rgensen set $\Lambda(\mu)$ is the set of $\lambda>0$ such that there exists $\mu_{\lambda}\in \mathcal{M}(\R^n)$ such that  $S(\mu)=S(\mu_{\lambda})$ and $(B_{\mu})^{\lambda}=B_{\mu_{\lambda}}.$ Denote $F_{\lambda}=F(\mu_{\lambda}).$ Recall that the set $(F_{\lambda})_{\lambda\in \Lambda(\mu)}$ is called by J{\o}rgensen an exponential dispersion model.
It is not correct to think that if $\mu^*$ is a dual measure of $\mu$  then  $\Lambda(\mu)= \Lambda(\mu^*).$  We will see in Section 5
that a dual measure $\mu^*$ of the Bernoulli distribution $  \mu$ exists and is infinitely divisible. Therefore
$$\Lambda(\mu^*)=(0,\infty)\neq \Lambda(\mu)=\{1,2,\ldots\}.$$
However if $\lambda$ is both in $ \Lambda(\mu) $ and $ \Lambda(\mu^*) $, then  a dual measure of  $\mu_{\lambda}$ satisfies \begin{equation}\label{JORG}B_{(\mu_{\lambda})^*}(m)=B_{\mu^*}(\frac{m}{\lambda}).\end{equation}

\subsection{Duality and the change of generating measure of an exponential family}

Let us clarify now what happens to $\mu^*$ and to $F(\mu^*)$ when we replace $\mu(dx)$ by
$\mu_1(dx)=e^{\<s_0,x\>}\mu(dx).$  This is important, since $F(\mu)=F(\mu_1).$

\begin{proposition}\label{prop:3.4}
Let $\mu\in \mathcal{M}(\R^n)$ and suppose that there exists a dual measure $\mu^*.$ Let $\mu_1(dx)=e^{\<s_0,x}\>\mu(dx).$ Then $\mu_1^*=\mu^**\delta_{s_0}$ is a dual measure of $\mu_1.$ In particular, the elements of $F(\mu_1^*)$ are the translation by $s_0$ of the elements of $F(\mu^*).$ Symmetrically if $\mu_2=\mu*\delta_{m_0}$ is a translation of $\mu$ by $m_0$  then
$\mu_2^*(dx)=e^{\<m_0,y\>}\mu^*(dy)$ is a dual measure of $\mu_2.$ In particular  $F(\mu_2^*)= F(\mu^*).$
\end{proposition}
We skip the proof, which is an immediate application of the definitions.  Let us consider  the simple example of the gamma exponential family $F$ with shape parameter $\lambda, $ which is Case 2 in the Tweedie scale example in Section 2.  The domain of the means is $M_F=(0,\infty)$ and its variance function is $V_F(m)=m^2/\lambda.$ Applying Proposition \ref{prop:3.1}
we write 
\begin{eqnarray*}&&\ell''_{\mu^*}=\frac{\lambda}{m^2}, \qquad \ell'_{\mu^*}=-\frac{\lambda}{m}+s_0, \qquad  \ell_{\mu^*}=-\lambda \log m+s_0m+s_1, \qquad B_{\mu^*}(m)=\frac{e^{s_0m+s_1}}{m^{\lambda}}=e^{s_1}\int_{0}^{\infty}e^{-(y-s_0)m}\frac{y^{\lambda-1}}{\Gamma(\lambda)}dy.
\end{eqnarray*}
If $s_0=0$ then  $F(\mu^*)$ is the same gamma family. If $s_0\neq 0$ this is a translation of this last family.
Suppose  that we have started from a translated gamma family with $M_F=(m_0,\infty)$ and $V_F(m)=\lambda/(m-m_0)^2.$ For this case we get the same families $F(\mu^*)$: they do not depend on the particular translation by $m_0.$

\subsection{Duality for T--exponential families} 
The preceding discussion shows that while the dual of a measure is well defined up to a multiplicative constant, this is not the case for an exponential family $F,$  since  changing the generating measure $\mu$ of $F$   into some $\mu_1$ implies that $F(\mu)=F(\mu_1)$ by definition, but possibly changes $F(\mu^*)$ into a translate $F(\mu_1^*)$. For this reason we coin the  definition of a T--exponential family:
\begin{definition}
Given an exponential family $F$ the associated T--exponential family $T\hskip-2pt F$ is the union of all the translations of $F$. If $F=F(\mu)$ then
$$T\hskip-2ptF=\{ P(-s,\mu)*\delta_{m_0}\ ;\ s\in S(\mu),\ m_0\in \R^n\}.$$
Such a $\mu\in \mathcal{M}(\R^n)$ is called a generating measure of $T\hskip-2pt F$.
\end{definition}
Note that $T\hskip-2pt F(\mu)=T\hskip-2pt F(\mu_1)$ if and only if there exists $s\in S(\mu), b\in \R^n, m_0\in \R^n$ such that $\mu_1=\delta_{m_0}*e^{-\<s_0 ,x\>+b}\mu(dx)$
or $$\ell_{\mu_1}(s)=-\<s,m_0\>+b+\ell_{\mu}(s+s_0).$$
We are  now in position to clearly define the dual $T\hskip-2pt F^*$ of a T--exponential family $T\hskip-2pt F$, thanks to the following proposition:
\begin{proposition}\label{prop:3.5}
If $\mu $ is a generating measure of the T-exponential family T\hskip-2pt F in $ \R^n$ and if $\mu$  has a dual measure $\mu^*$ then any generating measure $\mu_1$ has also a dual and
$$T\hskip-2pt F(\mu^*)=T\hskip-2pt F(\mu_1^*)$$
Under these circumstances we denote $T\hskip-2pt F^*=T\hskip-2pt F(\mu^*).$
\end{proposition}

For instance we have seen that if $T\hskip-2pt F$ is the T--exponential family of all the gamma distributions with shape parameter $\lambda$ augmented with  all its translations, then $T\hskip-2pt F=T\hskip-2pt F^*$ is self dual. We will find a similar phenomena with the Gaussian distributions with  variance one, then with two  unexpected  examples: below in Section 6  with the dilogarithm exponential family and the $ \sinh m$ family, and, as a generalization of the one dimensional gamma case, with the Wishart distributions with fixed shape parameter.

\subsection{Duality and large deviations.}  Let $F=F(\mu)$ be a real natural exponential family and let $m_0<m$ two points of $M_F.$ Let $X_1,\ldots,X_n,\ldots$ be independent random variables with the same distribution in $F$ with mean $m_0$. The theorem of large deviations, due to Cram\'er \cite {CRAMER}, says that for $$h(m_0,m)=\int_{m_0}^m\frac{m-t}{V_F(t)}dt$$ we have \begin{equation}\label{LARGEDEV}
\left(\Pr(\frac{1}{n}(X_1+\cdots+X_n)>m)\right)^{1/n}\overset{n\to \infty}\longrightarrow e^{-h(m_0,m)}.
\end{equation}
In the next proposition, we link $h(m_0,m)$ with $\mu^*$ as the rest of the Taylor expansion of $m\mapsto \ell_{\mu^*}(m)$ around $m_0$. It shows that $m\mapsto \exp h(m_0,m)$ defines a member of $T\hskip-2pt F^*$ when the distribution $P$ of $X_n$ is in $T\hskip-2pt F$. One can consider such a proposition as a kind of probabilistic interpretation of duality in dimension one.

\begin{proposition}\label{prop:3.6}
Let $T\hskip-2pt F$ be a real T--exponential family such that its dual $T\hskip-2pt F^*$ exists. Let $X_1,\ldots, X_n,\ldots$ be independent random variables with the same distribution $P\in T\hskip-2pt F$ and mean $m_0$. For $m_0<m$ define
$$h(m_0,m)=-\lim_{n\to \infty}\frac{1}{n}\log \Pr(\frac{1}{n}(X_1+\cdots+X_n)>m).$$  Suppose that the dual probability $P^*$ of $P$ exists.  Let $s_0=\ell_{P^*}(m_0)$, $ c=\ell_{P^*}(m_0)-m_0\ell'_{P^*}(m_0)$ and $P^*_{m_0}=e^{c}P^* *\delta_{s_0}.$  Then

\begin{equation}\label{LINK}B(P^*_{m_0})(m)=\int_{\R}e^{-my}P^*_{m_0}(dy)=e^{h(m_0,m)}.\end{equation}
\end{proposition}
\begin{proof}[\textbf{\upshape Proof:}]
From Proposition \ref{prop:3.1} we have $\ell''_{P^*}(m)=1/V_{F(P)}(m).$ Therefore by integration by parts 
\begin{eqnarray*}&&
h(m_0,m)=\int_{m_0}^m(m-t)\ell''_{P^*}(t)dt=(m_0-m)\ell'_{P^*}(m)+\ell_{P^*}(m)-\ell_{P^*}(m_0)=\ell_{P^*}(m)-ms_0+c.
\end{eqnarray*} 
\end{proof}

\subsection{Warnings for existence and non existence of the dual measures}
To conclude Section 3, devoted to  generalities on duality, we consider  some devices  for proving that a function is not a Laplace transform for the case of dimension 1. Other ones will appear in Section 7 for the $\R^n$ case.

It will be shown that the Vinogradov-Paris distribution has no dual. Vinogradov and Paris \cite{VOVA} in their Theorem 6 consider the exponential family $F$ on $(0,\infty)$ with variance function $2m^2/(1-m^2)$ defined on $M_F=(0,1).$ Now we show that it has no dual. Indeed if a dual measure $\mu^*$ exists we have from Proposition \ref{prop:3.1}
$$\ell''_{\mu^*}(m)=\frac{1-m^2}{2m^2}\Rightarrow B_{\mu^*}(m)=e^{-m^2/2}/\sqrt{m}.$$
The function $e^{-m^2/2}$ is a  Fourier transform, and $1/\sqrt{m}$ is a Laplace transform. Who wins? By the theorem of maximal analyticity of Laplace transforms (see Letac and Mora \cite{LETAC1990} and Kawata  \cite{KAWATA}) $B_{\mu^*}(m)$ is not only defined  on $M_F=(0,1)$ but also on $(0,\infty).$ However the second derivative of $\log B_{\mu^*}(m)$  is negative on $m>1$ and this proves that $B_{\mu^*}$ cannot be the Laplace transform of a positive measure.  For proving that some $T\hskip-2pt F^*$ does not exists, this trick can be utilized when $F$ is not steep.

Let us apply Proposition \ref{prop:3.5} to the  probability distributions
$$P_1(dx)=\d e^{-|x|}dx,\qquad P_2(dx)=\frac{dx}{2\cosh (\d \pi x)}$$
and to a distribution $P_3$ described in Letac \cite{RIO} generating the exponential family with variance function $(1+m^2)^{3/2}.$ The probabilities
$P_1$ and $P_2$ are called the bilateral exponential and the hyperbolic distribution. They have means $m_0=0$, and respective Laplace transforms
$$B_1(s)=\frac{1}{1-s^2},\qquad B_2(s)=\frac{1}{\cos s},$$ 
which generate two exponential families with respective variance functions
$$V_1(m)=m^2+1+\sqrt{m^2+1},\qquad V_2(m)=m^2+1.$$
Performing the calculation of $h(0,m)=\ell_{\mu*}(m)$ for $P_1,P_2,P_3$, we obtain
\begin{eqnarray}
\label{HONE}&&h_1(m)=2\frac{m^2+1-\sqrt{m^2+1}}{m^2}e^{\sqrt{m^2+1}-1},\qquad h_2(m)=\frac{1}{\sqrt{m^2+1}}e^{ m\arctan m},\qquad h_3(m)=e^{\sqrt{m^2+1}-1}.\end{eqnarray}

For seeing that $P_1^*$ ,  $P_2^*$  and $P_3^*$ do not exist we use Mathematica for computing
\begin{eqnarray*}&&
h_1(m)=1+\frac{3m^2}{4}+o(m^5),\qquad
h_2(m)=1+\frac{m^2}{2}+\frac{m^4}{24}+\frac{m^6}{80}-\frac{283 m^8}{3020}+o(m^9),\qquad
h_3(m)=1+\frac{m^2}{2}+o(m^5).
\end{eqnarray*}
This proves that $\int_{\R}x^4P_1^*(dx)=\int_{\R}x^4P_3^*(dx)=0,\ \ \int_{\R}x^8P_2^*(dx)<0$: none of them is possible. Thus $P_1^*$,  $P_2^*$, and  $P_3^*$ do not exist.  My thanks go to Lev Klebanov for suggesting this method for proving  the lack of duality of some measures.

\section {The  variance function $e^m$, the Landau law $\varphi$ and the dual of Poisson family} In this section, we restrict ourselves to the one dimensional case.
\subsection{The Landau distribution $\varphi$.} We will study $T\hskip-2pt F^*$ when $T\hskip-2pt F$  is the set of Poisson distributions augmented with its translations. In Section 5 we are going to study the  dual pairs $(T\hskip-2pt F,T\hskip-2pt F^*) $  issued of the case where $V_F(m)$ is a quadratic or  a cubic polynomial. For studying their dual we will use a particular probability $\varphi $ on $\R$  called the Landau law, which can be defined by its Laplace transform  defined on $S(\varphi)=(0, \infty)$ by $B_{\varphi}(s)=e^{-s}s^s$.

As we are going to see, $\varphi$ exists and has $\R$ as support. Sometimes the term of Landau distribution is given to the convolution $\varphi *\delta_{-1}$, because its Laplace transform has the more elegant form $s^s$. It is the law of $ X-1$ when $X\sim \varphi.$ A detailed study of the Landau
distribution can be found, for instance, in Marucho et al. \cite{Marucho}.

The law $\varphi$  generates the exponential family with variance function $e^{m}.$ For seeing this we write using for instance \eqref{MODIF2}:
\begin{equation}\label{LANDAU}\ell_{\varphi}(s)=s\log s-s,\qquad \ell'_{\varphi}(s)=\log s =-m,\qquad \ell''_{\varphi}(s)=\frac{1}{s}=e^{m}.\end{equation}
An important feature of $\varphi$ for our purposes  is the existence of the following  dual $\varphi^*$:
$$s=-\ell'_{\varphi}(-\ell'_{\varphi^*}(s))=\log \ell'_{\varphi^*}(s), \qquad \ell'_{\varphi^*}(s)=e^{-s}, \qquad  \ell_{\varphi^*}(s)=e^{-s}-1,\qquad  B_{\varphi^*}(s)=e^{e^{-s}-1}.$$ Thus a dual of the Landau distribution is the Poisson distribution with mean 1.

\subsection{Existence of $\varphi$ and parents}

\begin{proposition}\label{prop:4.1}
For $s>0$ and $R>1$ we have
\begin{eqnarray}\label{LAPLACELANDAU}
&&\hskip-1cm e^{-s} s^s=\exp \int_0^{\infty}(e^{-sx}-1+sx \, e^{-x})x^{-2}dx,\qquad
(1+s)^{1+s}(1-s)^{1-s}=\exp\int_{\R}(e^{-sx}-1-sx)x^{-2}e^{-|x|}dx,\\&&\hskip-1cm 
\label{DUALNEGBIN}\tfrac{s^s}{(s+1)^{s+1}}=\exp- \int_0^{\infty}(1-e^{-sx})(1-e^{-x})x^{-2}dx,\qquad
\tfrac{s^s}{(s+1)^{s}}=\exp- \int_0^{\infty}(1-e^{-sx})(1-e^{-x}(1+x))x^{-2}dx,\\&&\hskip-1cm 
\label{DUALTAKACS}\tfrac{s^s(s+1)^{\tfrac{s+1}{R-1}}}{(Rs+1)^{\tfrac{Rs+1}{R-1}}}
=R^{-\tfrac{Rs}{R-1}}\exp- \int_0^{\infty}(1-e^{-sx})f(x)dx,\qquad
f(x)=\tfrac{R-1+e^{-x}-Re^{-x/R}}{(R-1)x^2},\\&&\hskip-1cm 
\label{DUALRESSEL}\tfrac{1}{4}\left(\tfrac{s+2}{s+1}\right)^{s+2}=\exp- \int_0^{\infty}(1-e^{-sx})(x-1+e^{-x})x^{-2}dx.
\end{eqnarray}

\end{proposition}

\begin{proof}[\textbf{\upshape Proof:}]
We prove first that for $a>0$ we have
$$F(a)=\int_0^{\infty}(e^{-u}-1+u e^{-au})u^{-2}du=-1-\log a.$$
For seeing this we note that $ F'(a)=-1/a$ and doing an obvious integration by parts we observe that $F(1)=-1.$  We pass from the value of $F(a)$ to the first statement in \eqref{LAPLACELANDAU} by the change of variable $u=sx.$

For proving the second statement in \eqref{LAPLACELANDAU} we apply the first one while replacing $s$ by $1+s$ and $1-s$, thus obtaining two integrals. In the second one we change $x$ into $-x$ and thereafter summing the two integrals yield the second expression in \eqref{LAPLACELANDAU} plus the extra term
$$2+\int_{\R}(e^{-|x|}-1-|x|e^{-|x|})\frac{dx}{x^2}=0,$$ which is obtained by an integration by parts.

For proving the first relation in \eqref{DUALNEGBIN} we write patiently

\begin{eqnarray*}&&\hskip-1cm
\log \frac{s^s}{(s+1)^{s+1}}=s+\int_0^{\infty}(e^{-sx}-1+sx \, e^{-x})x^{-2}dx
-(s+1)-\int_0^{\infty}(e^{-sx}e^{-x}-1+sx \, e^{-x}+xe^{-x})x^{-2}dx
\\
&&=-1+\int_0^{\infty}(e^{-sx}(1-e^{-x})-xe^{-x})x^{-2}dx
=-1+\int_0^{\infty}(e^{-sx}-\tfrac{xe^{-x}}{1-e^{-x}})(1-e^{-x})x^{-2}dx
\\&&=-1+\int_0^{\infty}(e^{-sx}-1)(1-e^{-x})x^{-2}dx+\int_0^{\infty}(1-\tfrac{xe^{-x}}{1-e^{-sx}})(1-e^{-x})x^{-2}dx
=- \int_0^{\infty}(1-e^{sx})x^{-1}(1-e^{-x})x^{-1}dx.\end{eqnarray*}
The last line uses the fact observed above that $F(1)=-1.$
 For proving the second statement in \eqref{DUALNEGBIN} 
 we use the Frullani integral 
 \begin{equation}\label{FRULLANI}
 \log (1+s)=\int_0^{\infty}(1-e^{-sx})e^{-x}x^{-1}dx
 \end{equation}  
  that we add to $\log \frac{s^s}{(s+1)^{s+1}}$ for obtaining the desired result.

The proof of \eqref{DUALTAKACS} follows the same lines: we start from the log of \eqref{LAPLACELANDAU} and replace $s$ by $s+1$ and $s+\frac{1}{R}$. We watch the coefficient $f(x)$
of $e^{-sx}$ defined by \eqref{DUALTAKACS} and we rearrange the remainder terms. The integral $$\int_{0}^{\infty}f(x)dx=\frac{\log R}{R-1}$$ is computed by an integration by parts followed by the application of \eqref{FRULLANI}.

For proving \eqref{DUALRESSEL} we combine \eqref{FRULLANI} and the first statement in \eqref{DUALNEGBIN} applied to $s+1$ instead of $s$:
\begin{eqnarray*}\tfrac{1}{4}\left(\tfrac{(s+2)}{s+1}\right)^{s+2}&=&\tfrac{1}{4}\left(\tfrac{(s+2)^{s+2}}{(s+1)^{s+1}}\right)\tfrac{1}{s+1} =  \tfrac{1}{4}\exp \int_0^{\infty}\{(e^{-sx}-1)e^{-x}x^{-1}+(1-e^{-sx}e^{-x})(1-e^{-x})x^{-2}\}dx\\
&=& \tfrac{1}{4}\exp\Big\{ \int_0^{\infty}(1-e^{-sx})(x-1+e^{-x})x^{-2}dx+\int_0^{\infty}\left(\tfrac{1-e^{-x}}{x}\right)^2dx\Big\}.\end{eqnarray*}
To conclude we have by an integration by parts and \eqref{FRULLANI} that $\int_0^{\infty}\left(\frac{1-e^{-x}}{x}\right)^2dx=\log 4$. This ends the proof of \eqref{DUALRESSEL}. 
\end{proof}

\subsection{L\'evy measures.} Recall that we concentrate on the one dimensional case in this section. Extension to the $\R^n$ case is possible, but we will have no use of it in this paper.
Recall also that if a probability $P$ is in $\mathcal{M}(\R)$ then $P$ is infinitely divisible if
there exists a positive measure $\nu(dx)$ on $\R\setminus\{0\}$
such that $$\int_{\R\setminus\{0\}}\min(1,x^2)\nu(dx)<\infty,$$
 and there exist two numbers $a\in \R$ and $\sigma\geq 0$
such that
for all $s\in S(P)$ we have \begin{equation}\label{INFDIV}
\ell_P(s)=as+\d\sigma^2 s^2+\int_{\R\setminus\{0\}}(e^{-sx}-1+s\tau(x))\nu(dx).
\end{equation}
Here $\tau$ is a bounded function such that $\tau(x)/x\overset{x\to 0}\longrightarrow 1$. Feller \cite{FELLER} chooses $\tau(x)=\sin x$, the Russian school chooses $\tau(x)=x/(1+x^2).$ For getting simple formulas here we have chosen $\tau(x)=xe^{-|x|}.$ Changing $\tau$ may change $a$ but neither  $\sigma$  nor the measure $\nu$, which  is called the L\'evy measure of $P$. The number $\sigma$ is called the Gaussian part. If $\nu$ is bounded,  we say that $P$ or $\nu$ is Type 0. If $\nu$ is unbounded, but
$\int_{\R\setminus\{0\}}\min(1,|x|)\nu(dx)<\infty,$ we say that $P$ is Type 1. For Types 0 and 1, the representation of the Laplace transform $B_P(s)$ does not need $\tau$ and we can write
$$\ell_P(s)=as+\d\sigma^2 s^2-\int_{\R\setminus\{0\}}(1-e^{-sx})\nu(dx).$$
In other cases  we say that $P$ is of Type 2 and $\tau$ is necessary. The convex hull of the support of such  a probability of Type 2 is always $\R$.
\begin{example}
In Section 2 we have mentioned in the Tweedie scale the unsymmetric stable distribution with parameter $\gamma\in (1,2)$ defined by its Laplace transform
\begin{equation}\label{STABLE}e^{s^{\gamma}}=\exp \left(as+\int_0^{\infty}(e^{-sx}-1+sxe^{-x})x^{-\gamma-1}\frac{dx}{\gamma(\gamma-1)\Gamma(2-\gamma)}\right).\end{equation}
This formula can be guessed by differentiating $s^{\gamma}$ twice, and can be checked by two integrations by parts, yielding also the exact value of the constant $a.$ The L\'evy measure is of Type 2 and equals $$\nu(dx)=1_{(0,\infty)}(x)x^{-\gamma-1}\frac{dx}{\gamma(\gamma-1)\Gamma(2-\gamma)}.$$

From Proposition \ref{prop:4.1}, the first formula in \eqref{LAPLACELANDAU}, we see that the Landau distribution $\varphi$ exists, is infinitely divisible without Gaussian part and has $\nu(dx)=1_{(0,\infty)}(x)\frac{dx}{x^2}$ for L\'evy measure.  It is of Type 2.

From Proposition \ref{prop:4.1}, 
the second formula in \eqref{LAPLACELANDAU},
we see that there exists an infinitely divisible distribution $P$ such that $B_{P}(s)=(1+s)^{1+s}(1-s)^{1-s}$ with $S(P)=\R$ and  with L\'evy measure $$\nu(dx)=\frac{e^{-|x|}}{x^2}dx.$$ Its type is 2. We will see in Section 6 that $P$ is the dual of the symmetric binomial distribution $\d(\delta_{-1}+\delta_1) $ as expected from \eqref{SURPRISE} of the introduction.

From Proposition \ref{prop:4.1}, the first formula in \eqref{DUALNEGBIN}, we see that there exists an infinitely divisible distribution $P$ such that $B_{P}(s)=s^s/(s+1)^{s+1}$ with $S(P)=(0,\infty)$ and  with L\'evy measure $$\nu(dx)=1_{(0,\infty)}(x)\frac{1-e^{-x}}{x^2}dx.$$ Its type is 1. We will see later that $P$ is the dual of the negative binomial distribution.

From Proposition \ref{prop:4.1}, the second formula in \eqref{DUALNEGBIN}, 
we see that there exists an infinitely divisible distribution $P$ such that $B_{P}(s)=s^s/(s+1)^{s}$ with $S(P)=(0,\infty)$ and  with L\'evy measure $$\nu(dx)=1_{(0,\infty)}(x)\frac{1-e^{-x}(1+x)}{x^2}dx.$$ Its type is 0. We will see later that $P$ is the dual of an Abel distribution sometimes called generalized Poisson distribution, generating an exponential family on non-negative integers with cubic variance function $m(1+m)^2.$

From Proposition \ref{prop:4.1}, the first formula in \eqref{DUALTAKACS}, we see that there exists an infinitely divisible distribution $P$ such that $B_{P}(s)=\frac{s^s(s+1)^{\frac{s+1}{R-1}}}{(Rs+1)^{\frac{Rs+1}{R-1}}}$ with $S(P)=(0,\infty)$ and  with L\'evy measure $$\nu(dx)=1_{(0,\infty)}(x)f(x)dx$$ where the positive function $f$ is defined by \eqref{DUALTAKACS}. Its type is 0, since the integral of $f$ is finite.  We will see later that $P$ is the dual of a Tak\'acs   distribution, generating an exponential family on non negative integers with cubic variance function $m(1+m)(1+Rm)$.

From Proposition \ref{prop:4.1}, the second formula in \eqref{DUALTAKACS},
we see that there exists an infinitely divisible distribution $P$ such that $B_{P}(s)=(\frac{s+2}{s+1})^{s+2}$ with $S(P)=(0,\infty)$ and  with L\'evy measure $$\nu(dx)=1_{(0,\infty)}(x)\frac{x-1+e^{-x}}{x^2}dx.$$ Its type is 0. We will see later that $P$ is the dual of a Kendall-Ressel  distribution, generating an exponential family with cubic variance function $m^2(1+m)$.
\end{example}

\section{A dictionary of dual pairs
 for quadratic and cubic families in $\R$}
We split the section in two  parts: duals of Morris families, duals of cubic families.  The case of the Tweedie scale has been already ruled out in Section 3.1 and in particular, by  Proposition  \ref{prop:3.2}.  Most of the considered distributions are infinitely divisible, which is saying that the corresponding J\o rgensen set is $(0,\infty).$ Formulas for passing from $\mu$ to $\mu_{\lambda}$ and from $\mu^*$ to $(\mu_\lambda)^*$ have been given above in Section 3.2. For this reason,  except in the binomial case, we give simplified versions, ignoring the J\o rgensen parameter $\lambda$.
For instance, looking at the dual of the negative binomial family with variance function $m+\frac{m^2}{\lambda}$, we deal only with $m+m^2$ in order to have more readable formulas for the duals.

The presentation of each pair $(T\hskip-2pt F,T\hskip-2pt F^*)$  is done by selecting a $\mu$ and a $\nu$ which may be only a translation of $\mu^*$.   These measures are  generating $T\hskip-2pt F$ and $T\hskip-2pt F^*$ respectively, namely
$F(\mu)\subset T\hskip-2pt F, \ F(\nu)\subset T\hskip-2pt F^*.$ We give one of the two  variance functions $V_{F(\mu)}(m)$  and $V_{F(\nu)}(m)$ as well one of the two Laplace transforms $B_{\mu}(s)$ and $B_{\nu}(s)$. In most of the cubic cases, the densities on $\R$, or the weights in the discrete cases, are not available and either the Laplace transform or the variance function is computable explicitly: hence question marks may appear in the description of the dual. 

Logic would have imposed to write $V_{F(\nu)}(s)$ and $B_{\nu}(m)$ systematically. However we do not stick always to this rule. There are two reasons: variance functions of exponential families have become familiar objects during the last forty years from Morris  \cite{MORRIS}, with many examples, and we are used to write $m$ not $s$ for the mean, exactly like we are reluctant to call $c$ the radius of a circle and $r$ its center. The second reason is after all that $T\hskip-2pt F^{**}=T_F$. Which is the first, $T\hskip-2pt F$ or $T\hskip-2pt F^*?$

\subsection{Duals of Morris families}
\begin{enumerate}\item \textsc{The Poisson case} $$(m,\exp(e^{-s}); e^m, e^{-s}s^s).$$ The Poisson exponential family is generated by
$\mu=\sum_{n=0}^{\infty}\frac{\delta_n}{n!}$ with Laplace transform $\exp (e^{-s})$.  We have seen in Section 3.1 that
$\ell_{\varphi}(m)=m\log m-m$  and therefore $\ell'_{\varphi}(m)=\log m=-s$, $m=e^{-s}$ $\ell'_{\varphi^*}(s)=-e^{-s}$ and, ignoring the integration constant,
$\ell_{\varphi^*}(s)=e^{-s}.$  Therefore $\mu$ is a dual of the Landau distribution $\varphi. $

\item \textsc{The Bernoulli case} $$(m-m^2,  1+e^{-s}; 4(\cosh \frac{m}{2})^2, s^s(1-s)^{1-s}).$$
The Bernoulli distribution is generated by $\mu=\delta_0+\delta_1$ leading to \begin{equation}\label{LBI}\ell(\mu)=\log (1+e^{-s}),\ \ell'(\mu)(s)=-\frac{1}{1+e^s}=-m, \end{equation} where $m$ is in the domain of the means $(0,1)$ of the Bernoulli distribution.  Therefore
$\ell'_{mu^*}(m)=-s=-\log(1-m)+\log m$ and, ignoring the integration constant: $$\ell_{\mu^*}(m)=(1-m)\log(1-m)+m\log m, \  B_{\mu^*}(m)=m^m(1-m)^{1-m}.$$   One can describe $\mu^*$ by introducing the law $\tilde{\varphi}$ of $X$ when $- X$ has the Landau distribution $\varphi$, thus with Laplace transform $e^m(-m)^{-m}$ defined for $m<0.$ Now we introduce the measure $\varphi_1(dx)=e^{x-1}\tilde{\varphi}(dx)$ with Laplace transform, defined for $m<1$, 
$$ B_{\varphi_1}(m)=\int_{\R}e^{-mx}\varphi_1(dx)= e\int_{\R}e^{-mx+x}\tilde {\varphi}(dx)=e^{m}(1-m)^{1-m}.$$
Note, by assuming $m=0$ in $B_{\varphi_1}(m)$, $\varphi_1$ is a probability. Finally from the observation of their Laplace transforms we get a dual probability of $\mu$ as the convolution
of $\varphi$ and $\varphi_1.$ We obtain
$$\mu^*=\varphi*\varphi_1, \qquad B_{\varphi*\varphi_1}(m)=e^{-m}m^m e^{m}(1-m)^{1-m}=m^m(1-m)^{1-m},$$ 
since the product of two Laplace transforms is well defined on the intersection of their existence domain (here $(-\infty,1)$ and $(0,\infty)$. For computing the variance  function we use Proposition \ref{prop:3.1} and \eqref{LBI}: 
$$\frac{1}{V_{mu^*}(m)}=\ell''_{\mu}(m)=\frac{d}{dm}\frac{-1}{1+e^m}=\frac{4}{\cosh^2 \frac{m}{2}}.$$ 
In Proposition \ref{prop:4.1} and the second statement in \eqref{LAPLACELANDAU} we have considered the case with the symmetric binomial distribution $\mu=\frac{1}{2}(\delta_{-1}+\delta_1)$ which is an affine transformation of the ordinary Bernoulli distribution. We have seen that its dual exists and has Laplace transform $(1-s)^{1-s}(1+s)^{1+s}$, and is infinitely divisible with L\'evy measure  of Type 2. Since $\ell_{\mu}(s)=\log \cosh s$ we get  the variance function with Proposition \ref{prop:3.1}:
$$V(m) =\frac{1}{\ell''_{\mu}(m)}=\cosh ^2m.$$
\item \textsc{The  binomial case} 
\begin{equation}\label{SYMBERNDUAL}(m-\frac{m^2}{N},  (1+e^{-s})^N; \frac{4}{N}(\cosh \frac{m}{2})^2, s^s(N-s)^{N-s}).\end{equation}
For describing the dual of the binomial distribution from the study of the Bernoulli case, enough is to use \eqref{JORG}.

\item \textsc{The negative binomial  case}$$(m+m^2,  (1-e^{-s})^{-1}; 4(\sinh \frac{m}{2})^2, s^s(1+s)^{-1-s}).$$
The negative binomial is generated by $\mu=\sum_{n=0}^{\infty}\delta_n$ leading to 
\begin{equation}\label{LBI}\ell_{\mu}(s)=-\log (1-e^{-s}),\ \ell'_{\mu}(s)=-\frac{1}{e^s-1}=-m, \end{equation} 
where $m$ is in the domain of the means $(0,\infty)$. Therefore
$\ell'_{\mu^*}(m)=-s=-\log(1+m)+\log m$ and, ignoring the integration constant $$\ell_{\mu^*}(m)=-(1+m)\log(1+m)+m\log m, \  B_{\mu^*}(m)=m^m(1+m)^{-1-m}.$$   Proposition \ref{prop:4.1}, the first relation in \eqref{DUALNEGBIN}, has shown the existence of the probability $\mu^*$. It is an infinitely divisible distribution of Type 1. Since  $\ell'_{\mu}(s)=-\frac{1}{e^s-1}$, the variance function of $F(\mu^*)$ is easily obtained:
$$V_{ F(\mu^*}(m) =\frac{1}{\ell''_{\mu}(m)}= 4(\sinh \frac{m}{2})^2.$$

\item \textsc{The gamma  case}$$(\frac{m^2}{\lambda },  s^{-\lambda}; \frac{m^2}{\lambda },  s^{-\lambda}).$$
This self dual has been detailed in Section 3 as a particular case of the Tweedie scale.

\item \textsc{The Gaussian  case}$$(\sigma^2,  \exp(\d \sigma^2s^2); \sigma^{-2},  \exp(\d \sigma^{-2}s^2)).$$
This case is very simple  since $\ell_{N(0,\sigma^2)}(s)=\d \sigma^2s^2,\ \ell'_{N(0,\sigma^2)}(s)=\sigma^2s$ and therefore
$\ell'_{N^*(0,\sigma^2)}(s)=\sigma^{-2}s.$
\item \textsc{The hyperbolic case} We have seen in Section 3.6 with equation \eqref{HONE} that these laws have no dual.
\end{enumerate}

\subsection{Duals of cubic  families}
\begin{enumerate} \item \textsc{The Abel  case} $$(m(1+m)^2,  ?; ?,  s^s/(s+1)^{s}).$$
\item \textsc{The Takacs   case} $$(m(1+m)(1+Rm),  ?; ?, \frac{s^s(s+1)^{\frac{s+1}{R-1}}}{(Rs+1)^{\frac{Rs+1}{R-1}}} .$$
\item \textsc{The Kendall Ressel   case} $$(m^2(1+m),  ?; ?,  (\frac{s+1}{s})^{s+1}).$$

\end{enumerate}
The above given three cases are similar. The densities of these three cubic families are explicit (see Letac and Mora \cite{LETAC1990}) but not the Laplace transforms. Similarly the variance function of the dual is not computable (except if we accept a description by a sum of a series obtained by the Lagrange formula).
However using Proposition \ref{prop:3.1}  and the relation $\ell''_{\mu*}(m)=1/V_F(m)$ we can catch the Laplace transform of the dual. Therefore we apply this principle to the three functions 
$$\frac{1}{m(1+m)^2},\qquad  \frac{1}{m(1+m)(1+Rm)},\qquad \frac{1}{m^2(1+m)}.$$
Fortunately these relations can be integrated twice leading to the three Laplace transforms appearing in the second expression of \eqref{DUALNEGBIN}, \eqref{DUALTAKACS} and \eqref{DUALRESSEL}, respectively. For seeing that these three duals do exist, the hard work has been done in Proposition \ref{prop:4.1}, where it is shown that these three duals are infinitely divisible, as commented on in Section 4.3. About \eqref{DUALRESSEL}, observe that the Laplace transform $(\frac{s+1}{s})^{s+1}$ of an unbounded measure $\mu$ is changed  by the substitution $s\mapsto s+1$ into  the Laplace transform  $(\frac{s+2}{s+1})^{s+2}$,
then converted into the Laplace transform $\frac{1}{4}(\frac{s+2}{s+1})^{s+2}$ of a probability of $F(\mu)$:

\begin{enumerate}
\item \textsc{The Inverse Gaussian  case}
$$(\frac{m^3}{
\lambda},  e^{-\lambda \sqrt{2s}}; 2^{3/2}\lambda m^{3/2},  \exp{\lambda^2/4s}).$$
This is an application of Proposition \ref{prop:3.2} with $q=3$.

\item \textsc{The strict arcsine family  case}$$(m(1+m^2),  ?; ?, (\frac{s}{\sqrt{1+s^2}})^se^{\arctan s}).$$
For computing the Laplace transform of $\mu^*$ we use Proposition \ref{prop:3.1},
$$\ell''_{\mu^*}(s)=\frac{1}{s(1+s^2)}=\frac{1}{s}-\frac{s}{1+s^2},\qquad   \ell'_{\mu^*}(s)=\log s-\d \log (1+s^2),$$
leading to $B_{\mu^*}(s)=(\frac{s}{\sqrt{1+s^2}})^se^{\arctan s}.$ For seeing that the dual exists we note that
$$\int_0^{\infty}e^{-sx}(1-\cos x)dx= \frac{1}{s}-\frac{s}{1+s^2}  \ell_{\mu^*}(s)=\int_0^{\infty}(e^{-sx}-1)\frac{1-\cos x}{x^2}dx.$$
Therefore $\mu^*$ exists,  is infinitely divisible with L\'evy measure $\frac{1-\cos x}{x^2}1_{(0,\infty)}(x)dx$  and is of Type 0.

\item \textsc{The large  arcsine family  case} $(m(1+2m\cos a+m^2),?;?, B_{\mu^*}(s))$  with $0<a<\pi/2.$ The Laplace transform  of the dual Proposition \ref{prop:3.1} is obtained as follows:
$$\ell''_{\mu^*}(s)=\frac{1}{s(1+2s\cos a+s^2)}=\frac{1}{s}-\frac{1}{\sin^2 a}\frac{s+2\cos a}{\left(\frac{s+\cos a}{\sin a}\right)^2+1}.$$
Observing  that $\ell''_{\mu^*}$ is the Laplace transform of a positive measure is not quite obvious. Introducing $u=\frac{s+\cos a}{\sin a}$ leads to $s=u\sin a-\cos a$ and to
\begin{eqnarray*}\frac{1}{\sin^2 a}\frac{s+2\cos a}{\left(\frac{s+\cos a}{\sin a}\right)^2+1}&=&\frac{1}{\sin^2 a}\frac{u\sin a+\cos a}{u^2+1}=\frac{1}{\sin^2 a}\int_0^{\infty}e^{-ut}\sin (a+t)dt\\&=&\frac{1}{\sin^2 a}\int_0^{\infty}e^{-t\frac{s+\cos a}{\sin a}}\sin (a+t)dt=\frac{1}{\sin a}\int_0^{\infty}e^{-sx}e^{-x\cos a}\sin(a+x\sin a)dx.
\end{eqnarray*}
We now show that  for $x>0$ and $0\leq a\leq \pi/2$ we have $$\frac{1}{\sin a}e^{-x\cos a}\sin(a+x\sin a)\leq 1$$ or, equivalently, that $f(x)=\sin x\   e^{x\cos a}-\sin(a+x\sin a)\geq 0$. For seeing this we use the two inequalities $e^c\geq 1+c$ and $r\geq \sin r$ respectively applied to $c=x\cos a$ and $r=x\sin a$ and we get
\begin{eqnarray*}&&f(x)\geq \sin a (1+x\cos a)-\sin a \cos r-\sin r \cos a=\sin a(1-\cos r)+\cos a(r-\sin r)\geq 0.\end{eqnarray*}
As a consequence, denoting $$g(x)=1-\frac{1}{\sin a}e^{-x\cos a}\sin(a+x\sin a)\overset{x\to 0}\sim\frac{x^2}{2}$$ we can claim that
$\mu^*$ exists,  is infinitely divisible with L\'evy measure $\frac{g(x)}{x^2}1_{(0,\infty)}(x)dx$  and is Type 0.
One can note that the limit values $a=\pi/2$ and $a=0$ yield the Abel and the strict arcsine cases. The computation of the function  $B_{\mu^*}(s)$ is a painful exercise of calculus. We get indeed
\begin{eqnarray*}
\ell_{\mu^*}(s)&=&s
\log s-\tfrac12 (s+\cos a)\log (s+\cos a)+(-\tfrac12+(\tfrac{\pi}{2}-a)\cos a)s
-\tfrac12(\mathrm{cotan}\, a)\log (1+2s\cos a +s^2)
\\&&+(\mathrm{cotan}\, a)\left(\frac{s+\cos a}{\sin a}\right)\arctan\left(\frac{s+\cos a}{\sin a}\right)+\tfrac12 \cos a\log \cos a+(\mathrm{cotan}\, a)^2(\tfrac{\pi}{2}-a).
\end{eqnarray*}

\end{enumerate}
\section {The variance function $e^m-1$ and the dilogarithm  law}
In this section we still continue our study of duality in the one dimensional case. We introduce the dilogarithm law which generates the exponential family with variance function $e^m-1.$ This gives the opportunity to introduce parent distributions $\mu_r,$ $\sigma, $ $\sigma_r,$  $\eta$ and $\alpha $ with interesting properties, and exponential families with unexpected variance functions, like $e^m+1$ and $\sinh m.$ We will see that $\mu$ and $\alpha$ are self dual.
\subsection{The  generating probability $\mu$} From the Bar Lev criteria described in Letac  and Mora \cite{LETAC1990} (Corollary 3.3 and Proposition 4.4),  $V(m)=e^m-1$ defined on $(0,\infty)$ is the variance function of an exponential family $F$ concentrated on the non-negative integers. Let us compute a particular generating probability $\mu=\sum_{n=0}^{\infty}\mu(n)\delta_n$ of $F$. We use
$$ds=-\frac{dm}{V(m)}=-\frac{dm}{e^m-1}=-\frac{e^{-m}dm}{1-e^{-m}}=-d\log(1-e^{-m}).$$ Therefore $S(\mu)=(0,\infty)$ and
\begin{equation}\label{DUALDILOG}e^{-s}=1-e^{-m}\Rightarrow \ell'_{\mu}(s)=-m=\log (1-e^{-s})\Rightarrow \ell'_{\mu}(s)=-\sum_{n=1}^{\infty}\frac{e^{-sn}}{n}.\end{equation} Recall that
$$\sum_{n=1}^{\infty}\frac{1}{n^2}=\frac{\pi^2}{6}.$$
Continuing the calculation and choosing properly the integration constant such that $k_{\mu}(0)=0$ we have
$$ \ell_{\mu}(s)=-\frac{\pi^2}{6}+\sum_{n=1}^{\infty}\frac{e^{-ns}}{n^2}\Rightarrow B_{\mu}(s)=\sum_{n=0}^{\infty}\mu(n)e^{-ns}=\exp\left(-\frac{\pi^2}{6}+\sum_{n=1}^{\infty}\frac{e^{-ns}}{n^2}\right).$$
As a consequence
\begin{equation}\label{GENERATOR}\sum_{n=0}^{\infty}\mu(n)z^n=\exp\left(\sum_{n=1}^{\infty}\frac{z^n-1}{n^2}\right).\end{equation}
This generating function is linked to the classical special function $Li_2(z)$, called dilogarithm, defined on the unit disk $\{z\in \C\ ; |z|\leq 1\}$ by $Li_2(z)=\sum_{n=1}^ {\infty}\frac{z^n}{n^2},$ for which a vast literature exists from Euler: let us quote Lewin \cite{LEWIN}, Zagier \cite{ZAGIER} and Kirilov \cite{KIRILOV}. The measure $\mu$ is a probability as we can see by putting $z=1$ in \eqref{GENERATOR}.  It is an infinitely divisible distribution  with L\'evy mesure $\sum_{n=1}^{\infty}\frac{\delta_n}{n^2}$. For this reason we call $\mu$ the dilogarithm law in the sequel. Here is a surprising property that $\mu$ shares with the normal and the gamma laws, namely the self duality.
\begin{proposition}\label{prop:6.1}
The dilogarithm distribution $\mu$ is dual of itself.
\end{proposition}
\begin{proof}[\textbf{\upshape Proof:}]
The result is immediate from \eqref{DUALDILOG} which implies the symmetric relation $e^{-s}+e^{-m}=1$.
\end{proof}
Consider $Y$ and $Y'$ which are independent with the same distribution $\mu$. We denote by $\sigma$ the distribution of $Y-Y'$ in $\Z$. Consider also an element $P(-s,\mu))$ of the exponential family $F(\mu)$. Since $S(\mu)=(0,\infty)$ we rather write $r=e^{-s}\in (0,1)$ and  we consider the probability $\mu_r=P(-s, \mu)$ which satisfies
\begin{equation}\label{GENMUR}\S \mu_r(n)z^n=e^{Li_2(rz)-Li_2(r)}.\end{equation}
Similarly, we introduce  $Y_r$ and $Y_r'$ independent with the same distribution $\mu_r$ and denote by $\sigma_r$ the distribution of $Y_r-Y_r'$ in $\Z$.
In Sections 6.2 and 6.3 we study
the four probabilities $\mu, \sigma, \mu_r,  \sigma_r.$

\subsection{ Some  properties of $\mu$ and $\sigma$}
For having a probabilistic interpretation of the probability $\mu$ defined by \eqref{GENERATOR}, consider the probability $\nu$ on positive integers  defined by $$\nu=\frac{6}{\pi^2}\sum_{n=1}^{\infty}\frac{\delta_n}{n^2}$$ Suppose that $X_1,\ldots, X_n,\ldots$ are iid with distribution $\nu$  and consider an independent Poisson random variable $N$ with mean $\lambda=\pi^2/6.$ Then the distribution of $X_1+\cdots+X_N$  is $\mu$ by a classical calculation using conditioning by $N.$ Of course, this shows that $\mu$ is infinitely divisible of type 0.
\begin{proposition}\label{prop:6.2}
If $Y$ and $Y'$ are independent with the same distribution $\mu$ then the characteristic function of the random variable $Y-Y'\sim \sigma$ on $\Z$  is the function with period $2\pi$ given for $0\leq t\leq 2\pi$ by
$\E(e^{it(Y-Y')})=e^{-\frac{1}{2}t(2\pi-t)}$.
\end{proposition}
\begin{proof}[\textbf{\upshape Proof:}]
Since $\E(e^{it Y})=\exp\left(\sum _{n=1}^{\infty}\frac{e^{int}-1}{n^2}\right)$ we have $$\E(e^{it(Y-Y')})=\exp\left(2\sum _{n=1}^{\infty}\frac{\cos nt-1}{n^2}\right)=\exp\left(-\frac{1}{2}t(2\pi-t)\right),$$ and due to elementary calculation of Fourier series we have, for $0\leq t\leq 2\pi$, $\sum _{n=1}^{\infty}\frac{\cos nt}{n^2}=\frac{\pi^2}{6}-\frac{1}{4}t(2\pi-t)$.
\end{proof}
\begin{remark}
1) Since $S(\mu)=(0,\infty )$, then $\sigma$ has no Laplace transform and cannot generate an exponential family.
2) The calculation of $\Pr(Y=Y')$ is not elementary. More specifically by the change of variable $s=t-\pi$ and Mathematica:
$$Pr(Y=Y')=\frac{1}{2\pi}\int_0^{2\pi}e^{-\frac{1}{2}t(2\pi-t)}dt=\frac{1}{\pi}\int_0^{\pi}e^{\d s^2-\d \pi^2}ds=0.11751.$$
By a similar calculation for $n\in \Z$
$$Pr(Y-Y'=n)=\frac{1}{2\pi}\int_0^{2\pi}e^{-\frac{1}{2}t(2\pi-t)+int}dt=\frac{(-1)^n}{\pi}\int_0^{\pi}e^{\d s^2-\d \pi^2}\cos (ns)ds.$$
\end{remark}
The fact that $\E(e^{it(Y-Y')})=e^{\frac{t^2}{2}-\pi t}$ in $[0,2\pi]$ calls for a link  with the Gaussian distribution, as it is shown in the next proposition.
\begin{proposition}\label{prop:6.3}
 If $Z\sim N(0,1)$ and  $Y,Y'\sim \mu$ are independent, then the logarithm of the characteristic function of the continuous random variable $W=Z+Y-Y'$  is made of piecewise  affine functions and is concave. More specifically, if $t-2\pi k\in [0,2\pi)$ with $k\in \Z$ then
$$\log \E(e^{itW})=2k(k+1)\pi^2-(2k+1)\pi t$$  and the density $f(x)$ of $W$ is, using the notation $e_k(x)=e^{-2k^2\pi^2-2i\pi x}$, equals
\begin{equation}\label{ELLIPTIQUE}
f(x)=\frac{1}{2\pi}\sum_{k\in \Z}\frac{e_k(x)-e_{k+1}(x)}{(2k+1)\pi+ix}.\end{equation}
\end{proposition}
\begin{proof}[\textbf{\upshape Proof:}] If $t-2\pi k\in [0,2\pi)$ with $k\in \Z$ we have
$$\log \E(e^{itW})=-\frac{1}{2}t^2+\frac{1}{2}(t-2k\pi)^2-\pi(t-2k\pi)=2k(k+1)\pi^2-(2k+1)\pi t.$$   
To compute $f(x)$ we use the inverse Fourier transform leads easily to \eqref{ELLIPTIQUE}:
\begin{eqnarray*}f(x)=\frac{1}{2\pi}\int_{-\infty}^{\infty}e^{-itx}\E(e^{itW})dt=\frac{1}{2\pi}\sum_{k\in \Z}\int_{2k\pi}^{2(k+1)}e^{-itx+2k(k+1)\pi^2-(2k+1)\pi t}dt.
\end{eqnarray*}
\end{proof}
\subsection{ Some  properties of $\mu_r$ and $\sigma_r$}
\begin{proposition}\label{prop:6.3}
Let $0<r<1$. If $Y_r$ and $Y_r'$ are independent with the same distribution $\mu_r$ then the  Laplace transform of the law $\sigma_r$ of $Y_r-Y'_r$ is defined for $s$ in $S(\sigma_r)=(\log r,\log(1/r)$ by $$\E(e^{-s(Y_r-Y'_r)})=\exp \left(-\int_0^r\log (1-2x\cosh  \theta+x^2)\frac{dx}{x}-2\int_0^r\log (1-x)\frac{dx}{x}\right).$$
The distribution $\sigma_r$ generates an exponential family $F(\sigma_r)$, where $\R$ is the domain of the means and the  variance function is, with $a=\sinh \frac{m}{2}$,
$$V_{F(\sigma_r)}(m)=\frac{2}{1-r^2}\sqrt{r^2+a^2}(\sqrt{r^2+a^2}+\sqrt{1+a^2}).$$
\end{proposition}
\begin{proof}[\textbf{\upshape Proof:}] From \eqref{GENMUR} $B_{\sigma_r}(s)$ exists if and only if $re^{-s}<1$ and $re^{-s}<1,$ in other terms if $|s|<-\log r.$ Since $Li_2(r)=-\int_0^r\log(1-x )\frac{dx}{x}$ we get
$$Li_2(re^{-s})=-\int_0^{re^{-s}}\log(1-x )\frac{dx}{x}=-\int_0^r\log(1-xe^{-s} )\frac{dx}{x}$$ and $B_{\sigma_r}(s)$ is easily deduced from this expression.
\end{proof}
For computing the variance function, we observe that
\begin{equation}\label{MOY}m=k'_{\sigma_r}(s)= \log (1-re^{-s})- \log (1-re^{s}).\end{equation}
Equality \eqref{MOY} shows that $\R$ is the domain of the means by letting $s\to \pm \log r.$ Also,  \eqref{MOY} can  be rewritten,
\begin{equation}\label{SINH}a=\sinh \frac{m}{2}=r\sinh\left( -s+\frac{m}{2}\right)\end{equation}
With the classical notations $\theta=\psi (m)$ and $V_{F(\sigma_r)}(m)=1/\psi'(m)$ we can write, by differentiating \eqref{SINH} with respect to $m$,
\begin{eqnarray*}&&\tfrac12\sqrt{1+a^2}=\tfrac12 \cosh \tfrac{m}{2}=r\cosh\left( \psi(m)+\tfrac{m}{2}\right)(\psi'(m)+\tfrac12)=r\sqrt{1+\frac{a^2}{r^2}}\left(\frac{1}{V_{F(\sigma_r)}(m)}+\tfrac12\right).\end{eqnarray*} This transform  leads to the desired formula:
$$V_{F(\sigma_r)}(m)=\frac{2\sqrt{r^2+a^2}}{\sqrt{1+a^2}-\sqrt{r^2+a^2}}=\frac{2}{1-r^2}\sqrt{r^2+a^2}(\sqrt{1+a^2}+\sqrt{r^2+a^2}).$$
The calculation of the dual of $\sigma_r$ involves elliptic integrals and its existence is unproven.
\subsection{The variance function $e^m+1$} In this section we study the exponential family $F(\eta)$ on $\R$ with variance function $e^m+1$
and with the most surprising property that probability $\eta$ is equal to $N(0,1)*\mu$ where $\mu$ is the dilogarithm probability defined by \eqref{GENERATOR}. 

Since the function $e^m+1$ is real analytic and positive on $\R$ then, if it is a variance fuction, the domain of the means of the corresponding natural exponential family will be $\R.$ Now, to decide whether it is a variance function or not, we have to compute a potential Laplace transform in the usual way:

$$ds=-\frac{dm}{e^m+1}=\frac{ e^{-m}dm}{1+e^{-m}}=-\frac{ du}{1+u}=-d\log(1+u)=-d\log(1+e^{-m}).$$ Since $m>0$ we can take   $S(\eta)= (0,\infty)$ and $s=\log (1+e^{-m}).$ This leads to

 $$\ell'_{\eta}(s)=-m=\log (e^{s}-1)=s+\log (1-e^{-s}) =s-\sum_{n=1}^{\infty}\frac{e^{-ns}}{n}.$$
By choosing properly  integration constant such that $\eta$ is a probability, we get
$$\ell_{\eta}(s)=\tfrac12 s^2-\tfrac{\pi^2}{6}+\sum_{n=1}^{\infty}\frac{e^{-ns}}{n}=\ell_{N(0,1)}(s)+\ell_{\mu}(s).$$
Showing the unexpected result that $\eta=N(0,1)*\mu$ proves also that $\eta$ exists and is a positive measure, and that $e^m+1$  is the variance function of an exponential family. Proposition \ref{prop:6.3} given above can be reformulated in terms of $\mu$ and $\eta$. If $X\sim \eta$ and $Y'\sim \mu$
 are independent, then $W\sim X-Y'$.
One can remark that since $e^m$ and $e^{m}\pm 1$ are variance functions, therefore using affinities and the J{\o}rgensen set would enable us to describe all exponential families with variance functions $Ae^{m/\lambda} +B.$ We leave to the reader to prove that $\eta$
has no dual. 
\subsection{The variance function $\sinh(m)$} In this section we study the exponential family $F(\alpha)$ on $\N$ with variance function $\sinh(m)$
and we describe the link between the probability $\alpha=(\alpha(n))_{n=0}^{\infty}$ and the dilogarithm function $Li_2(z).$ The existence of $\alpha$ is granted by the Bar-Lev theorem since the Taylor expansion of $\sinh(m)$ has only nonnegative coefficients (see  \cite{LETAC1990}, Corollary 3.3). The fact that $\alpha $ is concentrated on $\N$ is granted by the fact that $\sinh(m)\overset{m\to 0}\sim m$ (see  \cite{LETAC1990}, Proposition 4.4). For computing $\alpha$ we proceed in the usual way
\begin{eqnarray*}&&ds=-\frac{dm}{\sinh(m)}=-\frac{2 e^mdm}{e^{2m}-1}=-\frac{2 du}{u^2-1}=-\left(\frac{1}{u-1}-\frac{1}{u+1}\right)du=-d\log\frac{u-1}{u+1}=-d\log\frac{e^m-1}{e^m+1}.\end{eqnarray*} Since $m>0$ we can take   $S(\alpha)= (0,\infty)$ and $s=-\log \frac{e^m-1}{e^m+1}.$ This leads to
 \begin{equation}\label{ALPHA}\ell'_{\alpha}(s)=-m= -\log \frac{1+e^{-s}}{1-e^{-s}}=-2\sum_{n=1}^{\infty}\frac{e^{-(2n-1)s}}{2n-1}.\end{equation} Writing $z=e^{-s}\in (0,1)$ we get
$$\ell_{\alpha}(s)=C+ 2\sum_{n=1}^{\infty}\frac{z^{2n-1}}{(2n-1)^2}=C+2Li_2(z)-\frac{1}{2}Li_2(z^2).$$ Since $Li_2(1)=\pi^2/6$,  in order to have $\alpha$ of mass 1, we take $C=-\pi^2/4$ and we finally obtain
\begin{equation}\label{ALPHA2}\sum_{n=0}^{\infty}\alpha(n)z^n=e^{-\frac{\pi^2}{4}+2 Li_2(z)-\frac{1}{2}Li_2(z^2)}.\end{equation}
Of course, if $$\beta(dx)=\frac{8}{\pi^2}\sum_{n=1}^{\infty}\frac{1}{(2n-1)^2}\delta_{2n-1},$$  if $N$  is Poisson distributed with mean $\pi^2/4$,
and if $X_1,\ldots, X_n,\ldots$ are iid with distribution $\beta$ and independent of $N$, we have $X_1+\cdots+X_N\sim \alpha$.
\begin{proposition}\label{prop:6.5} The probability $\alpha$ is dual of itself.
\end{proposition}
\begin{proof}[\textbf{\upshape Proof:}] From \eqref{ALPHA} we obtain
$e^{-m}+e^{-s}+e^{-s-m}=1.$ The symmetry between $m$ and $s$ implies the result.
\end{proof}

The last proposition links $\alpha $ and $ \mu$. The proof follows from \eqref{ALPHA2}.

\begin{proposition}\label{prop:6.6}If $Y\sim \mu$ and $X\sim \alpha^{*2}$ are independent, then $X+2Y\sim \mu^{*4}$.
\end{proposition}

\section{Examples of  duality in $\R^n$} 
The classification of exponential families in $\R^n$ has been done in the literature with the same guidelines as for $\R$: from the simplest variance functions like the Tweedie scale, or  like the quadratic ones of Morris \cite{MORRIS} to more complicated ones like the cubic families  (see \cite{LETAC1990}) or the Babel class (see \cite{RIO}). For $\R^n,$ several  choices of classification through a definition of simple variances have been carried out:

\begin{enumerate}
\item The first one has been to extend the  gamma family with shape parameter $p>0$, whose variance function is $m^2/p$, to homogeneous quadratic variance functions in $\R^n.$ More specifically one had to find which variance matrices $V_F(m)$ of order $n$ are made of homogeneous quadratic polynomials with respect to $m=(m_1,\ldots,m_n).$ The answer has been given by Casalis \cite{CASALIS91}: out of trivial cases, the only exponential families with homogeneous quadratic variance functions are the Wishart ones, on symmetric, Hermitian, quaternionic Hermitian matrices, and analogous objects on the Lorentz cones and on the  exceptional Albert algebra.
\item The second choice has been to consider the so called simple quadratic families  in $\R^n.$ While a quadratic family has variance function of the form $V_F(m)=A(m)+B(m)+C$, where $m\mapsto A(m)$ is quadratic homogeneous, $m\mapsto B(m)$ is linear and $C$ is a constant symmetric matrix of order $n$, such a family is said to be simple if $A(m)$ has rank one, and therefore of the form $\frac{1}{\lambda} m\otimes m$;  recall that if $E$ is a Euclidean space and $a,b$ are in $E$ then  $a\otimes b$ is the endomorphism of $E$ defined by $$x\mapsto (a\otimes b)(x)=a\<b,x\>.$$ Here again their classification has been done by Casalis \cite{CASALIS}, while the classification of  general quadratic families is an open problem.
\item The third block is the simple cubic exponential families,  obtained from the simple quadratic ones by a M\"obius transformation of their variance function. They have been classified by Hassa\"iri \cite{HASSAIRI}.
\item The last choice is the so called diagonal families in $\R^n$ defined by the fact that the diagonal part of $V_F(m)$ has the form $(V_1(m_1),\ldots,V_n(m_n)).$ They have been classified in a paper with six coauthors: see Bar-Lev et al. \cite{BARLEVETAL}.
\end{enumerate}

\subsection{The Wishart families are self dual}
In the Euclidean space $V$ of real symmetric matrices of order $n$ with scalar product $\<x,y\>=\tr(xy)$ consider the cone $V_+\subset V$ of positive definite matrices and the cone $\overline {V_+}\subset V$ of semipositive definite matrices. If
$$p\in  \Lambda=\left\{\d, 1, \frac{3}{2},\ldots, \frac{n-1}{2}\right\}\cup \left(\frac{n-1}{2},\infty \right),$$  the family of the Wishart distributions of shape parameter $p$ is generated by the unbounded measure $\mu_p$ on $\overline {V_+}$ which can be defined on $s\in V_+$ by
 its Laplace transform $$\int_{\overline {V_+}}e^{-\<s,x\>}\mu_p(dx)=\frac{1}{(\det s)^p}.$$
 As a consequence $\ell_{\mu_p}(s)=-p\det s,\ \ell'_{\mu_p}(s)=-ps^{-1}$ since the gradient  of $s\mapsto \log \det s$ is $s^{-1}.$
Clearly $-\ell_{\mu_p}(-\ell_{\mu_p}(s))=p(ps^{-1})^{-1}=s$ and this shows that $\mu_p$ is a dual of itself.

We are not going to give details for the Wishart distributions defined on the other symmetric cones: after introducing the necessary definitions, the simple calculation above remains the same. One can consult Casalis and Letac \cite{CASALISLETAC}, for instance.

\subsection {The dual of the multinomial distribution}

Let $(e_1,e_2,\ldots,e_n)$ be the canonical basis of $\R^n$ and let us call Bernoulli distribution any law concentrated on the $n+1$  points $(0, e_1,e_2,\ldots,e_n),$ namely of the form
$$p_0\delta_0+p_1\delta_{e_1}+p_2\delta_{e_2}+\cdots+p_n\delta_{e_n},$$ 
where $p_i>0$ for $i\in\{0,\ldots,n\}$ and $p_0+p_1+\cdots+p_n=1.$
The set of all these Bernoulli distributions is an exponential family generated by the measure $\mu=\delta_0+\delta_{e_1}+\cdots+\delta_{e_n}$ whose Laplace transform is $$\Delta=B_{\mu}(s_1,\ldots,s_n)=1+e^{-s_1}+\cdots+e^{-s_n}$$ Therefore $$\ell_{\mu}(s)=\log \Delta,\qquad m_1=-\frac{\partial}{\partial s_1}\ell_{\mu}(s)=\frac{e^{-s_1}}{\Delta},\ldots,\ m_n=-\frac{\partial}{\partial s_n}\ell_{\mu}(s)=\frac{e^{-s_n}}{\Delta}.$$
For deciding whether $\mu^*$ exists or not we compute $\ell_{\mu^*}(m)$ such that $-\ell'_{\mu^*}(-\ell'_{\mu}(s)=s.$   Thus we have to compute $(s_1,\ldots,s_n)$ with respect to $(m_1,\ldots,m_n)$ as follows:
$$-m_1\Delta =   e^{-s_1},\ \ldots, -m_n\Delta =   e^{-s_n},\qquad 1+(m_1+\cdots+m_n)\Delta=\Delta,\qquad  \Delta=\frac{1}{1-m_1-\cdots-m_n}.$$
Therefore
$$s_1=-\log m_1+\log(1-m_1-\cdots-m_n),\ldots, s_n=-\log m_n+\log(1-m_1-\cdots-m_n).$$
As a consequence, if $\mu^*$ does exist, it must satisfy
$$\ell'_{\mu^*}(m)=\left(\log m_1-\log(1-m_1-\cdots-m_n), \ldots, \log m_n-\log(1-m_1-\cdots-m_n)\right).$$ It is easy to see that up to an additive constant we have
\begin{eqnarray*}\ell_{\mu^*}(m)&=&(m_1\log m_1 +\cdots+m_n\log m_n+(1-m_1-\cdots-m_n)\log(1-m_1-\cdots-m_n),\\ B_{\mu^*}(m)&=&m_1^{m_1} \ldots m_n^{m_n}(1-m_1-\cdots-m_n)^{1-m_1-\cdots-m_n}\end{eqnarray*}
Recall that the Landau distribution $\varphi$ is defined by $\int_{\R}e^{-sx}\varphi(dx)=e^{-s}s^s.$
In order to describe the above $\mu^*$ let us coin a simple lemma:

\begin{lemma}\label{lem:7.1}
Let $f(s)=f(s_1,\cdots,s_n)=a_1s_1+\cdots+a_ns_n+b=\<a,s\>+b$ be a non-constant affine form of $\R^n$. Then
$e^{-f(s)}f(s)^{f(s)}$, defined on the half plane $H=\{s; f(s)>0\}$,  is the Laplace transform of a positive  measure $\mu(dx)$  on $\R^n$, concentrated on the line $\R a$, defined as the image of the measure $\varphi_1(dt)=e^{-bt}\varphi(dt)$ by the map $t\mapsto x= at.$ This measure is bounded if $b\geq 0$, of mass $e^bb^b$ when $b>0$, and mass 1 if $b=0.$
\end{lemma}
\begin{proof}[\textbf{\upshape Proof:}] By definition we have
$$\int_{\R^n}e^{-\<s,x\>}\mu(dx)=\int_{\R}e^{-t\<s,a\>}\varphi_1(dt)=\int_{\R}e^{-t\<s,a\>-bt}\varphi(dt)=\int_{\R}e^{-tf(s)}\varphi(dt).$$
The remainder is plain. 
\end{proof}
With  this lemma we can describe $\mu^*(dx)$ which satisfies
$$\int_{\R^n}e^{-s_1x_1-\cdots-s_nx_n}\mu^{*}(dx)=s_1^{s_1}\ldots s_n^{s_n}(1-s_1-\cdots-s_n)^{1-s_1-\cdots-s_n}$$ on the tetrahedron  $\{s; s_1,\ldots,s_n,1-s_1-\cdots-s_n>0\}$ which is the domain of the means of the exponential family of the Bernoulli distribution. From the lemma, applied to the $n+1$ affine forms $f_0(s)=-s_1-\cdots-s_n+1$, $f_1(s)=s_1,\ \ldots\ f_n(s)=s_n$,  the measure $\mu^*$ is the convolution of $n+1$ probabilities respectively concentrated on the lines $\R e_1,\ldots, R e_n$ and $\R(-e_1-\cdots-e_n).$ For $f_0$ the measure described in the lemma is bounded and has mass $e^{-1}$  and has to be normalized to become  a probability: $$e^{1-f_0}f_0^{f_0}e^{-f_1}f_1^{f_1}\ldots,e^{-f_n}f_n^{f_n}=(1-s_1-\cdots -s_n)^{1-s_1-\cdots-s_2}s_1^{s_1}\ldots s_n^{s_n}.$$
It is interesting to compute the variance function of $F(\mu^*):$

\begin{proposition}\label{prop:7.2}Let $\Delta=1+e^{-s_1}+\cdots+e^{-s_n}$, $D=\mathrm{diag}(e^{-s_1},\ldots,e^{-s_n})$
and $J_n$ the $(n,n)$ matrix with all entries equal to one. Let $\mu^*$ be  a dual of $\mu=\delta_0+\delta_{e_1}+\cdots+\delta_{e_n}.$ Then the variance function of $F(\mu^*)$ is

$$V_{F(\mu^*)}(s)=\Delta(D^{-1}+J_n)=(1+e^{-s_1}+\cdots+e^{-s_n})\left[\begin{array}{ccc}e^{s_1}+1&\ldots&1\\\vdots&\ddots&\vdots\\1&\ldots&e^{s_n}+1\end{array}\right].$$
\end{proposition}
\begin{proof}[\textbf{\upshape Proof:}]
Let $v=v(s)=(e^{-s_1},\ldots,e^{-s_n})^T$.
Note that $\Delta'=v$ and that $v'=-D.$ With this notation we have, since $\ell_{\mu}=\log \Delta$,

$$ \ell'_{\mu}(s)=-\frac{1}{\Delta}v, \qquad \ell''_{\mu}(s)=\frac{1}{\Delta}D-\frac{1}{\Delta^2}v\otimes v,\qquad   V_{F(\mu*)}(s)=[\ell''_{\mu}(s)]^{-1}.$$
For computing the inverse matrix we use the following identity: if $ a\in \R^n$  such that $\|a\|^2=a_1^2+\cdots+a_n^2$ is different from 1, then $$(I_n-a\otimes a)^{-1}=I_n+\frac{a\otimes a}{1-\|a\|^2}.$$ We are going to apply this expresion to $$a=\Delta^{-1/2}D^{-1/2}v=\Delta^{-1/2}(e^{-s_1/2},\ldots,e^{-s_n/2})^T,$$ which satisfies
$$\|a\|^2=\frac{\Delta-1}{\Delta}, \qquad 1-\|a\|^2=\frac{1}{\Delta},\qquad \frac{a\otimes a}{1-\|a\|^2}=(D^{-1/2}v)\otimes (D^{-1/2}v),$$
\begin{eqnarray*}V_{F(\mu*)}(s)&=&\left(\frac{1}{\Delta}D-\frac{1}{\Delta^2}v\otimes v\right)^{-1}
=\Delta D^{-1/2}(I_n-a\otimes a)^{-1}D^{-1/2}=\Delta D^{-1/2}(I_n+\frac{a\otimes a}{1-\|a\|^2})D^{-1/2}\\
&=&\Delta(D^{-1}+\Delta( D^{-1/2}a) \otimes ( D^{-1/2}a))=\Delta(D^{-1}+J_n).
\end{eqnarray*}
\end{proof}
\begin{remark}
Note that when $n=1$ the variance is $(1+e^{-s})(1+e^s)=4\cosh^2 \frac{s}{2}$ as seen in \eqref{SYMBERNDUAL} .
\end{remark}

\subsection{The other simple quadratic families.} Casalis \cite{CASALIS} splits the set of simple quadratic families in $\R^n$ in $2n+4$ types, extending the famous Morris classification for $n=1$ in six families. Among them $n+1$ are trivial from our point of view, with variance function of the form $B(m)+C$ where $m\mapsto B(m)$ is linear. In fact, with a proper choice of a basis of $\R^n$ they are product of $k$ Poisson families and $n-k$ one dimensional Gaussian families with $k\in\{0,\ldots,n\}$ (see also \cite{LETAC1989}). Since trivially a product of duals is dual of the product, there is nothing to add. The remaining $n+3$ families are made with two exceptional ones, the multinomial and the hyperbolic, and a set of $n+1$  types denoted $ (\mathrm{NM-ga})_k$ families with $k\in\{0,\ldots,n\}$ that we will describe later. In the sequel  we are going to prove that neither the hyperbolic type nor the $( \mathrm{NM-ga})_k$ have a dual when $k>0$.  The case $k=n$ is referred to as the multivariate negative binomial case. The case $k=0$ has a dual if and only if $n=2$ or $n=3.$ 

Now it is shown that there is no dual for the multivariate negative binomial distribution.
Denote by $\N$ the set of nonnegative integers.  Consider the measure on $\N^n$ defined by $\mu(dx)=(\delta_0-\delta_{e_1}-\cdots-\delta_n)^{-1*}$ where the inverse is taken in the sense of convolution. Its Laplace transform is $$B_{\mu}(s)=\frac{1}{1-e^{-s_1}-\cdots -e^{-s_n}}=\frac{1}{\Delta},$$ defined on the set $$S(\mu)=\{s\in \R^n; \Delta>0 \}.$$ We show that  no  dual $\mu^*$ of $\mu$ exists. Suppose the existence of $\mu^*$. Since $\ell_{\mu}=-\log\Delta$ we have
$$m=(m_1,\ldots,m_n)=-\ell'_{\mu}(s)=\frac{1}{\Delta}(e^{-s_1},\ldots,e^{-s_n}).$$
Since $e^{s_i} =m_i\Delta$ we get $\Delta (m_1+\cdots+m_n)=e^{-s_1}+\cdots+e^{-s_n}=1-\Delta$ leading to $\Delta=1/(1+m_1+\cdots+m_n).$
Finally  we have on the domain $m_i\geq 0$, for all $i\in\{1,\ldots,n\}$,
$$\ell'_{\mu^*}(m)=(-s_1,\ldots, -s_n)$$ with $-s_i=\log m_i-\log (1+m_1+\cdots+m_n)$ and therefore
\begin{eqnarray}\nonumber
\ell_{\mu^*}(m)&=&m_1\log m_1+\cdots+m_n\log m_n-(1+m_1+\cdots+m_n)\log (1+m_1+\cdots+m_n),\\
B_{\mu^*}(m)&=&\frac{m_1^{m_1}\cdots m_n^{m_n}}{(1+m_1+\cdots+m_n)^{1+m_1+\cdots+m_n}}.
\label{LAPLACEDUALNEGBI}\end{eqnarray}
For fixed $m$ and  $m'$  the matrix \begin{equation}\label{MATRIXM}M=\left[\begin{array}{cc}B_{\mu^*}(2m)&B_{\mu^*}(m+m')\\B_{\mu^*}(m+m')&B_{\mu^*}(2m')\end{array}\right]\end{equation} is semi positive definite since
\begin{eqnarray*}(a,b)M\left(\begin{array}{c}a\\b\end{array}\right)&=&a^2 B_{\mu^*}(2m)+2abB_{\mu^*}(m+m')+b^2B_{\mu^*}(2m')=\int_{\R^n}(ae^{-\<m,x\>}+be^{-\<m',x\>})^2\mu^{*}(dx)\geq 0.\end{eqnarray*} This implies that $\det M\geq 0.$ Let us apply this remark to the particular case $m=(0,1,0,\ldots,0)$ and $m'=(1,1,0,\ldots,0). $ It leads to
$$\det M=\frac{1}{3^34^75^5}(4^8-3^35^5)=-\frac{19339}{3^34^75^5}<0$$ which is the desired contradiction.

Next it is shown that there is no dual for the $( \mathrm{NM-ga})_{n-1}$ case. 
  We use the notation of \cite{CASALIS}. This exponential family is generated by the  measure $\mu(dx_1,\ldots,dx_{n-1},dy)$ on $\N^{n-1}\times (0,\infty)$ defined by

$$(\delta_0-\delta_{e_1}-\cdots-\delta_{e_{n-1}})^{-1*}(dx_1,\ldots,dx_{n-1}) \frac{y^{x_1+\cdots+x_{n-1}}}{(x_1+\cdots+x_{n-1})!}1_{(0,\infty)}(y)dy.$$
The exponential family $F(\mu)$ is the set of the laws of $(X,Y)$ when $X=(X_1,\ldots,X_{n-1})$ is multivariate negative binomial on $\N^{n-1}$
 and the conditional law of $Y$ knowing $X$ is gamma with shape parameter $1+|X|=1+X_1+\cdots+X_{n-1}.$ Its variance function is
$$V_{\mu}(m)= m\otimes m+\mathrm{diag}(m_1,\ldots,m_{n-1},0)$$ and the Laplace transform of $\mu$ is
$$B_{\mu}(s)=(s_{n}-e^{-s_1}-\ldots -e^{s_{n-1}})^{-1}=\frac{1}{\Delta}.$$
We show first the following:  If $\mu^*$ exists then its Laplace transform is defined on $m_i>0$, $i\in\{1,\cdots,n\}$, by
\begin{equation}
\label{NBGAS}B_{\mu^*}(m)=\frac{m_1^{m_1}\ldots m_{n-1}^{m_{n-1}}}{m_n^{m_1+\cdots+m_{n-1}+1}}.\end{equation}
To see this, since $\Delta'=(e^{-s_1},\ldots,e^{-s_{n-1}},1)$ and that $m=-\ell'_{\mu}(s)=\Delta'/\Delta$  one gets $m_n=1/\Delta$ and $e^{-s_i}=m_i/m_n$ for $i<n$ leading to
$$\ell'_{\mu^*}(m)=-s=\log( m_1/m_n),\ldots, \log( m_{n-1}/m_n), (1+m_1+\cdots+m_{n-1})/m_n$$
from which the result is obtained.

To see that actually $\mu^*$ does not exist we use the same method as for the multivariate negative binomial distribution by observing that
the matrix $M$ defined by \eqref{MATRIXM} has a negative determinant when choosing $m=(0,\ldots,0,1)$ and $m'=(1,\ldots,0,1).$

Now it will be established that there is no dual for the $( \mathrm{NM-ga})_k$ case with $0<k<n-1$.
This exponential family is generated by the  measure $\mu(dx_1,\ldots,dx_{k},dy,dz)$ on $\N^{k}\times (0,\infty)\times \R^{n-k-1}$ defined by

$$(\delta_0-\delta_{e_1}-\cdots-\delta_{e_{k}})^{-1*}(dx_1,\ldots,dx_{k}) \frac{y^{x_1+\cdots+x_{k}}}{(x_1+\cdots+x_{k})!}1_{(0,\infty)}(y)dy\times N(0,yI_{n-k-1})(dz),$$
where $N(0,yI_{n-k-1})$ is the Gaussian distribution on $\R^{n-k+1}$ with covariance matrix $yI_{n-k-1}$.
Its Laplace transform is $B_{\mu}(s)=1/\Delta$ with $$\Delta=s_{k+1}-\d\sum_{j=k+2}^n s_j^2-\sum_{i=1}^k e^{-s_i}.$$
Computations, quite analogous to the previous one, yield that the Laplace transform of $\mu^* $ would be

\begin{equation}\label{NMGAK}B_{\mu^*}(m)=\frac{\prod_{i=1}^ke^{-m_i}m_i^{m_i}}{m_{k+1}^{1+m_1+\cdots+m_k}}
 \exp\left(\d\sum_{j=k+2}^n\frac{m_j^2}{m_{k+1}}\right).\end{equation}
If $k\neq 0$, putting $m_{k+2}=\ldots=m_n=0$ shows that $B_{\mu^*}(m)$ cannot be a Laplace transform as we have seen for \eqref{NBGAS}.

Consider now the $( \mathrm{NM-ga})_0$ case
$B_{\mu}(s)=1/\Delta$ with $$\Delta=s_1-\d\sum_{j=2}^n s_j^2.$$ This generates  the exponential family containing the distribution of $(X_1,Z_2\sqrt{X_1},\ldots,Z_n\sqrt{X_1}),$ where $X_1,Z_2,\ldots,Z_n$ are independent, with $Z_i\sim N(0,1)$ and where $X_1$ is exponential with mean 1. Standard calculations show that if $\mu^*$ exists then
$$B_{\mu^*}(m)=\frac{1}{m_1}\exp\left(\d\sum_{j=2}^n\frac{m_j^2}{m_1}\right).$$
Denoting for simplicity $\vec{m}=(m_2,\ldots,m_n)$ we can write
$$B_{\mu^*}(m_1,\vec{m})=\int_{\R^{n-1}}e^{-\<\vec{m},x\>-\d m_1\|x\|^2}m_1^{(n-3)/2}\frac{dx}{(2\pi)^{(n-1)/2}}.$$
If $n-3\leq 0$ we prove that $B_{\mu^*}$ is a Laplace transform as follows:

If $n=2$ we use $$\frac{1}{\sqrt{m_1}}=\int_0^{\infty}e^{-t m_1}\frac{dt}{\sqrt{\pi t}}$$ for writing \begin{eqnarray*}
&&B_{\mu^*}(m_1,m_2)=\frac{1}{\sqrt{2\pi}}\int_{\R}e^{-m_2 x}\int_0^{\infty}e^{-m_1(t+\d x^2)}\frac{dt}{\sqrt{\pi t}}dx=\frac{1}{\sqrt{2\pi}}\int_{\R}e^{-m_2 x}\int_{x^{2}/2}^{\infty}e^{-m_1 y}\frac{dy}{\sqrt{\pi (y-\d x^2)}}dx.\end{eqnarray*}
This shows that $B_{\mu^*}$ is the Laplace transform of the density $\frac{1}{\sqrt{2\pi}\sqrt{\pi (y-\d x^2)}}$ restricted to the interior of the parabola $y=x^2/2.$

If $n=3$ we have $$B_{\mu^*}(m_1,\vec{m})=\tfrac1{2\pi}\int_{\R^{2}}e^{-\<\vec{m},x\>-\d m_1\|x\|^2}dx.$$ This is the Laplace transform of the following singular measure in $\R^3$: it is  concentrated on the cone $y=\d(x_1^2+x_2^2)$ and it is the image of the  Lebesgue measure $\frac{dx_1dx_2}{2\pi}$ on $\R^2$ by the map $x\mapsto (x,\d\|x\|^2).$

To complete the study of these last two  measures $\mu^*$, one can describe the variance function of $F(\mu^*)$ with the following matrices for $n=2$ and $n=3:$
\begin{eqnarray*}V_{F(\mu^*)}(s_1, s_2)&=&\left(s_1-\d s_2^2\right)\left[\begin{array}{cc}s_1+\d s_2^2&\d s_2\\\d s_2&1\end{array}\right],\\ V_{F(\mu^*)}(s_1, s_2,s_3)&=&\left(s_1-\d (s_2^2+s_3^2)\right)\left[\begin{array}{ccc}s_1+\d (s_2^2+s_3^2)&\d s_2&\d s_3\\\d s_2&1&0\\\d s_3&0&1\end{array}\right].\end{eqnarray*}
We skip this standard computation.

For $n\geq 4$ one can suspect that $\mu^*$ does not exist. Enough is to prove it for $n=4,$ since we can do $m_5=\ldots=m_n=0$ to pass from the case $n>4$ to the case $n=4.$ For simplicity of calculations in the case $n=4$ we rather set  $m_1=t$ and $m_2=s_1, \ m_3=s_2,\ m_y=s_3.$ And we show now that $$L(\vec{s},t)=\frac{1}{t+1}\exp\left(\frac{1}{2(t+1)}\|\vec{s}\|^2\right)$$ is not a Laplace transform on $(-1,\infty)\times \R^3.$
Suppose the contrary. Consider a random variable $(\vec{X},Y)$ such that for $t>-1$
$$A=e^{-\<\vec{s},\vec{X}\>-tY},\ \mathbb{E}(A)=L(\vec{s},t).$$
Denote $\1=(1,1,1)$ We are going to prove that for a suitable $(r,\vec{s},t)$ we have \begin{equation}\label{ULTIMA}\E(\|\vec{X}-rY\1\|^2A)<0\end{equation}
which proves the  impossibility. By taking partial derivatives of $L(\vec{s},t)$ a calculation gives 
$\E(\|\vec{X}-rY\1\|^2A)=\frac{1}{(t+1)^5}e^{\frac{u}{t+1}}B$
with the shorter notations $u=\d\|\vec{s}\|^2$ and $$B=2u(t+1)^2+2r\<\vec{s},\1\>((t+1)^3+(t+1)^2(1+u))+r^2(2(t+1)^2+4u(t+1)+u^2)$$
We now choose $\vec{s}=-2r\1$. This gives $u=6r^2$ and $\<\vec{s},\1\>=6r.$ The quantity $B$ becomes 
$$B=-12  r((t+1)^3+(t+1)^2(1+6r^2))+r^2(14(t+1)^2+24(t+1)r^2+36r^4).$$
For fixed $t$
we can now choose $r$ small enough such that $B<0$. Therefore  \eqref{ULTIMA} is proved.

Finally it is shown that there is no dual for  the multivariate hyperbolic  case.
Following Casalis \cite{CASALIS}, p.~1836, the multivariate hyperbolic case is described by the Laplace transform
$$B_{\mu}(s)=(\cos s_n-e^{-s_1}-\cdots-e^{s_{n-1}})^{-1}$$
 defined on $$S(\mu)=\{s\in \R^{n-1}\times (-\d \pi,\d\pi)\ ; \cos s_n-e^{-s_1}-\cdots-e^{s_{n-1}}>0\}.$$ We spare to the reader the proof of the fact that if a dual $\mu^*$ exists then its Laplace transform is, with the notation $S=1+m_1+\cdots+m_{n-1}$,
$$B_{\mu^*}(m)=\frac{\prod_{i=1}^{n-1}m_i^{m_1}}{(S^2+m_n^2)^{S/2}} \exp(m_n\arctan (m_n/S)).$$
Note that if $m_1=\cdots=m_{n-1}=0$ we get $B_{\mu^*}(m)=h_2(m_n)$ where $h_2$ is defined in \eqref{HONE}. But is has been proven in Section 3 that $h_2$ cannot be a Laplace transform. As a consequence the  present $\mu^*$ does not exist.

\section{Remarks.} Duality, with its links with the large deviations as described in Proposition \ref{prop:3.6}, has not a strong probabilistic interpretation compared  to the notion of reciprocity, which was used in Letac and Mora \cite{LETAC1990}. However it widens the  zoo of variance functions, and it creates unexpected links like the pair Poisson-Landau. The problem of deciding of the existence of a dual is unsolved in many circumstances, and we need more tools than we described in Section 3.5.
Why do some measures have dual and why others have not? The Babel class has been defined  and described in Letac \cite{RIO}, as the set  the variance functions $b\Delta+(am+c)\sqrt{\Delta}$  where $ \Delta$ is a polynomial of degree $\leq 2$. Its study from the duality view point  has not been considered in the present paper. The same problem arises with the elliptic variance functions (Letac \cite{LETAC2016}) or the Seshadri class (Kokonendji \cite{KOKO}).
  The deep functional equations of the dilogarithm described in \cite{{KIRILOV}, {LEWIN}, {ZAGIER}} should lead  to  new properties of the dilogarithm distribution. A complete characterization  of self duality would also be in order. The role of steepness is not well understood: all tractable examples of a non steep $\mu$ lead to the non existence of a dual . Does non steepness imply no dual? Dealing with exponential families in $\R^n$ is generally hard: the simplest non trivial diagonal family in $\R^2$ is generated by the measure $\delta_0+\delta_{e_1}+\delta_{e_2}+c\delta_{e_1+e_2}$ where $c\neq 1$ (see \cite{BARLEVETAL} p.~895) and leads to computations of great complexity. We have not tried yet the cubic families in $\R^n$ for $n>1$ characterized by Hassairi \cite{HASSAIRI}.

\section*{Acknowledgments}
I have had some discussions with  Shaul Bar-Lev, Lev Klebanov and Vladimir Vinogradov  about various parts of the paper. I thank them all.

\end{document}